\documentclass[10pt,american]{amsart}
\usepackage{babel}
\usepackage{amsmath}
\usepackage{amssymb}
\usepackage{hyperref}
\hypersetup{ colorlinks = true, urlcolor = blue, linkcolor = blue, citecolor = red }
\usepackage{xcolor}
\usepackage{enumitem}
\usepackage{esint}

\allowdisplaybreaks
\makeatletter
\newsavebox{\@brx}
\newcommand{\llangle}[1][]{\savebox{\@brx}{\(\m@th{#1\langle}\)}%
  \mathopen{\copy\@brx\kern-0.5\wd\@brx\usebox{\@brx}}}
\newcommand{\rrangle}[1][]{\savebox{\@brx}{\(\m@th{#1\rangle}\)}%
  \mathclose{\copy\@brx\kern-0.5\wd\@brx\usebox{\@brx}}}
\makeatother

\usepackage{varwidth}
\usepackage{tasks}
\DeclareMathOperator{\dv}{div}

\DeclareMathOperator*{\osc}{osc}

\newcommand{\RR}{\mathbb{R}}
\newcommand{\mA}{\mathcal{A}}
\newcommand{\Om}{\Omega}
\newcommand{\na}{\nabla}
\newcommand{\pa}{\partial}

\newcommand{\ep}{\epsilon}
\newcommand{\ka}{\kappa}

\newcommand{\la}{\lambda}
\newcommand{\sig}{\sigma}
\newcommand{\data}{\mathit{data}}
\renewcommand{\d}{\mathrm{d}}
\newcommand{\eps}{\varepsilon}
\renewcommand{\rho}{\varrho}
\newcommand{\bb}{B}
\newcommand{\bbs}{B_*}
\newcommand{\dd}{D}
\newcommand{\tb}{\ell}
\newcommand{\tc}{h}
\newcommand{\tq}{\widetilde{Q}}
\newcommand{\muplus}{\boldsymbol{\mu}^+}
\newcommand{\muminus}{\boldsymbol{\mu}^-}

\newcommand{\bomega}{\boldsymbol{\omega}}

\DeclareMathOperator*{\essosc}{ess\,osc}

\DeclareMathOperator*{\esssup}{ess\,sup}
\DeclareMathOperator*{\essinf}{ess\,inf}

\theoremstyle{plain}
\newtheorem{theorem}{Theorem}[section]
\newtheorem{lemma}[theorem]{Lemma}
\newtheorem{corollary}[theorem]{Corollary}

\newtheorem{definition}[theorem]{Definition}

\newtheorem{remark}[theorem]{Remark}

\def\Xint#1{\mathchoice
    {\XXint\displaystyle\textstyle{#1}}%
    {\XXint\textstyle\scriptstyle{#1}}%
    {\XXint\scriptstyle\scriptscriptstyle{#1}}%
    {\XXint\scriptscriptstyle\scriptscriptstyle{#1}}%
    \!\int}
\def\XXint#1#2#3{\setbox0=\hbox{$#1{#2#3}{\int}$}
    \vcenter{\hbox{$#2#3$}}\kern-0.5\wd0}

\def\dashint{\Xint{\raise4pt\hbox to7pt{\hrulefill}}}

\def\XXiint#1#2#3{\setbox0=\hbox{$#1{#2#3}{\iint}$}
    \vcenter{\hbox{$#2#3$}}\kern-0.5\wd0}

\def\Xint#1{\mathchoice
	{\XXint\displaystyle\textstyle{#1}}%
	{\XXint\textstyle\scriptstyle{#1}}%
	{\XXint\scriptstyle\scriptscriptstyle{#1}}%
	{\XXint\scriptstyle\scriptscriptstyle{#1}}%
	\!\int}
\def\XXint#1#2#3{{\setbox0=\hbox{$#1{#2#3}{\int}$}
		\vcenter{\hbox{$#2#3$}}\kern-.5\wd0}}

\def\YYint#1#2#3{{\setbox0=\hbox{$#1{#2#3}{\iint}$}
		\vcenter{\hbox{$#2#3$}}\kern-.51\wd0}}
 
\def\Xint#1{\mathchoice
	{\XXint\displaystyle\textstyle{#1}}%
	{\XXint\textstyle\scriptstyle{#1}}%
	{\XXint\scriptstyle\scriptscriptstyle{#1}}%
	{\XXint\scriptscriptstyle\scriptscriptstyle{#1}}%
	\!\int}
\def\XXint#1#2#3{{\setbox0=\hbox{$#1{#2#3}{\int}$ }
		\vcenter{\hbox{$#2#3$ }}\kern-.6\wd0}}

\def\dashint{\Xint-}

\DeclareMathOperator{\dist}{dist}
\usepackage{nameref}
\makeatletter
\let\orgdescriptionlabel\descriptionlabel
\renewcommand*{\descriptionlabel}[1]{%
	\let\orglabel\label
	\let\label\@gobble
	\phantomsection
	\edef\@currentlabel{#1}%
	\let\label\orglabel
	\orgdescriptionlabel{#1}%
}
\makeatother
\numberwithin{equation}{section}

\def\Xint#1{\mathchoice
    {\XXint\displaystyle\textstyle{#1}}%
    {\XXint\textstyle\scriptstyle{#1}}%
    {\XXint\scriptstyle\scriptscriptstyle{#1}}%
    {\XXint\scriptscriptstyle\scriptscriptstyle{#1}}%
    \!\int}
\def\XXint#1#2#3{\setbox0=\hbox{$#1{#2#3}{\int}$}
    \vcenter{\hbox{$#2#3$}}\kern-0.5\wd0}

\def\dashint{\Xint{\raise4pt\hbox to7pt{\hrulefill}}}

\def\XXiint#1#2#3{\setbox0=\hbox{$#1{#2#3}{\iint}$}
    \vcenter{\hbox{$#2#3$}}\kern-0.5\wd0}

\begin{document}
	
\title[H\"older regularity for degenerate equations]{H\"older regularity for degenerate parabolic double-phase equations}

\author{Wontae Kim}
\address[Wontae Kim]{Department of Mathematics, Aalto University, P.O. BOX 11100, 00076 Espoo, Finland}
\email[corresponding author]{wontae.kim@aalto.fi}
%
\author{Kristian Moring}
\address[Kristian Moring]{Fachbereich Mathematik, Paris-Lodron-Universit\"at Salzburg, Hellbrunner Str.~34, Salzburg, 5020,  Austria}
\email{kristian.moring@plus.ac.at}
%
\author{Lauri Särkiö}
\address[Lauri Särkiö]{Department of Mathematics, Aalto University, P.O. BOX 11100, 00076 Espoo, Finland}
\email{lauri.sarkio@aalto.fi}

\everymath{\displaystyle}

\makeatletter
\@namedef{subjclassname@2020}{\textup{2020} Mathematics Subject Classification}
\makeatother

\begin{abstract}
We prove that bounded weak solutions to degenerate parabolic double-phase equations of $p$-Laplace type are locally H\"older continuous. The proof is based on phase analysis and methods for the $p$-Laplace equation.
In particular, the phase analysis determines whether the double-phase equation is locally similar to the $p$-Laplace or the $q$-Laplace equation.

\end{abstract}

\maketitle

\section{Introduction}

We consider parabolic double-phase equations whose prototype is 
$$
\partial_t u-\dv \left( |\nabla u|^{p-2} \nabla u + a(x,t) |\nabla u|^{q-2}\nabla u \right)=0 
$$
for $1<p<q<\infty$ and a nonnegative H\"older continuous coefficient function $a$.
More precisely, in this paper we consider equations of type 
\begin{align}\label{double-phase}
	\partial_t u-\dv\mA(x,t,u,\na u)=0\quad \text{in}\quad \Om_T,
\end{align}
where the vector field $\mA$ satisfies appropriate structure conditions of double-phase type with $a\in C^{\alpha, \frac{\alpha}{2}}(\Omega_T,\RR_{\geq 0})$ defined in Section~\ref{sec:prelim}. Here we focus on the range 
\begin{equation} \label{eq:p-q-range}
2\le p<q \leq p+\alpha.
\end{equation}

The theory for elliptic double-phase problems is well understood. As relatively recent developments we mention that Harnack's inequality for minimizers of functionals with non-standard growth was shown in \cite{MR3348922}. Regularity properties for solution and its gradient were studied without boundedness assumption in~\cite{MR3294408}, and with boundedness assumption local $C^{1,\beta}$-regularity was proved in~\cite{MR3360738}. We stress the fact that the condition $q\le p + \alpha$ in~\cite{MR3360738} is the same as here in~\eqref{eq:p-q-range}, and the optimality of the range was demonstrated in \cite{MR2058167}. For related results in the elliptic case, we refer also to \cite{BB,BCM,MR3294408,MR4630451}.

In the parabolic case, gradient regularity has been studied in terms of higher integrability in~\cite{MR4627284,KS}. For questions on existence, uniqueness and assumptions on energy space, we refer to~\cite{CGZ,KKS,S}. We also mention that an approach towards Harnack's inequality has been suggested in~\cite{MR4357734}. However, many regularity questions are still open. In the present paper, we address such a problem by showing that weak solutions to~\eqref{double-phase} are H\"older continuous in the range~\eqref{eq:p-q-range}.

The results on local H\"older continuity for parabolic $p$-Laplace type equations go back to the results by DiBenedetto~\cite{DB-86} in the degenerate case and by DiBenedetto and Chen~\cite{DC} in the singular case. The treatment of both cases can be found in the monograph~\cite{DB-book}, in which the proof relies on DeGiorgi type iteration argument based on energy and logarithmic estimates for truncations of solutions. For more recent developments of the proof technique, we refer to~\cite{DGV-book} and~\cite{liao-unified}. Moreover, local H\"older continuity has been shown for several other nonlinear parabolic PDEs, e.g. for porous medium type equations in~\cite{DF,DGV-book,liao-unified} and for doubly nonlinear equations in~\cite{dn-part1,dn-part2,dn-part3}.

At this stage we state our main result. For a compact subset $K \subset \Omega_T$, we denote an intrinsic parabolic $q$-distance to the parabolic boundary $ \Gamma  = (\Omega \times \{0\}) \cup (\partial \Omega \times [0,T))$ by
$$         \dist_q(K,\Gamma)=
         \inf_{\substack{(x,t) \in K \\ (y,s) \in \Gamma  }} \left\{ |x-y| +(1+\|a\|_\infty)^\frac{1}{q} \|u\|_\infty^{\frac{q-2}{q}} |t-s|^\frac{1}{q} \right\},
$$
 
where we denoted $\|a\|_{\infty}=\|a \|_{L^\infty(\Om_T)}$ and $\|u\|_{\infty}=\|u \|_{L^\infty(\Om_T)}$.
 
\begin{theorem}\label{thm}
     Let $u$ be a bounded weak solution to \eqref{double-phase} according to Definition~\ref{def:weak-sol} such that~\eqref{as:structure} and~\eqref{as:range-a} are in force. Then $u$ is locally H\"older continuous in $\Om_T$. Moreover, there exist $c > 0$ and $\beta \in(0,1)$ depending only on $n,p,q,\alpha,C_0,C_1,[a]_\alpha$ and $\|u\|_\infty$ such that for any compact set $K \subset \Omega_T$
     $$
     |u(x_1,t_1)-u(x_2,t_2)|\le c\left(\frac{|x_1-x_2|+(1+\|a\|_\infty)^\frac{1}{q}\|u\|_{\infty}^{\frac{q-2}{q}}|t_1-t_2|^\frac{1}{q} }{\min\{1,\dist_q(K, \Gamma ) \}}\right)^\beta
     $$
     holds for every pair of points $(x_1,t_1),(x_2,t_2)\in K$. 
\end{theorem}

The overall strategy of the proof is based on measure theoretic alternatives together with DeGiorgi type iteration argument as in~\cite{DB-book,Urbano-book}. In order to apply this method, we pose a suitable criterion to be able to reduce the energy estimates of the equation to estimates of either the $p$-Laplace or the $q$-Laplace equation. 

The cylinders we consider are of the form
\begin{equation} \label{eq:cylinder}
Q_{\rho, \theta_\rho \rho^2} (x_0,t_0) = B_{\rho}(x_0) \times (t_0 - \theta_\rho \rho^2, t_0),
\end{equation}
in which
$$
\theta_\rho = \Big(\frac{\bomega}{\rho}\Big)^2 \biggl( \Big(\frac{\bomega}{\rho}\Big)^p + a(x_0,t_0) \Big(\frac{\bomega}{\rho}\Big)^q \biggr)^{-1},
$$
where $\bomega$ denotes an upper bound for the oscillation of the solution $u$ in  a cylinder containing the   cylinder~\eqref{eq:cylinder},  see the recursive argument in Corollary~\ref{recursive} for the details.   The quantity $\theta_\rho$ reflects an appropriate pointwise intrinsic scaling parameter with respect to the solution to the double-phase problem in our context. Moreover, we separate the cases where the energy estimates of~\eqref{double-phase} in Lemma~\ref{lem:caccioppoli} behave like a $p$-Laplace or a $q$-Laplace equation in a given cylinder, i.e., where
\begin{equation} \label{eq:cases}
a(x_0,t_0) \Big(\frac{\bomega}{\rho}\Big)^q \leq K \Big(\frac{\bomega}{\rho}\Big)^p,
\end{equation}
or the negation ($>$), respectively, holds true.
The constant $K$ in~\eqref{eq:cases} depends on the H\"older coefficient of $a$ as well as the quantity $\|u\|_\infty^{q-p}$, see~\eqref{eq:K}. With this choice we can further show that in the $q$-phase $a(x,t)$ is comparable to the value $a(x_0,t_0)$ in the considered cylinder, as shown in Lemma~\ref{phase_criterion}. Then the aim is to use~\eqref{eq:cases} and its alternative to transform Caccioppoli inequality, Lemma~\ref{lem:caccioppoli}, to correspond either parabolic $p$-Laplace or $q$-Laplace equation, respectively, in the iteration.

In this paper, we consider solutions which belong to the space $L^q(0,T;W^{1,q}(\Omega)) \cap L^\infty(\Omega_T)$ by assumption. The first condition on the Sobolev space property is a somewhat stronger assumption compared to what is considered to be the natural energy space, but we suppose it in order to avoid additional technical complications in deriving energy estimates in Section~\ref{sec:energy-est}. For relaxation of such an assumption, we encourage an interested reader to consult e.g.~\cite{KKS}. In connection to the second condition, we note that the local boundedness of solutions to the double-phase equation was studied in~\cite{MR3482400} in the context of problems with non-standard growth (see also \cite{MR3366102} for the elliptic case). In the present paper, we consider solutions that are bounded by definition, as it may occur that the range of $q$ in~\eqref{eq:p-q-range} is beyond the range considered in~\cite{MR3482400} when $\alpha$ is close to $1$.

\section{Preliminaries} \label{sec:prelim}

\subsection{Notations and definition of solution}
Let $\Om \subset \RR^n$ be an open set for $n\ge2$, $T > 0$ and denote $\Om_T =\Omega\times(0,T ] $. 
In~\eqref{double-phase} we suppose that $\mA(x,t,u,\na u):\Omega_T\times \RR \times \RR^{n}\longrightarrow \RR^{n}$ is a Carath\'eodory vector field satisfying the following structure assumptions: there exist constants $0<C_0\le C_1<\infty$ such that
	\begin{align}\label{as:structure}
 \begin{split}
     |\mA(z,v,\xi)|&\le C_1(|\xi|^{p-1}+a(z)|\xi|^{q-1}),\\
     \mA(z,v,\xi)\cdot \xi&\ge C_0(|\xi|^p+a(z)|\xi|^q)
 \end{split}
	\end{align}
for almost every $z\in \Omega_T$ and every pair $(v,\xi)\in \RR \times \RR^{n}$.

Throughout the paper, we assume that $2\le p<q<\infty$ and the non-negative coefficient function $a:\Omega_T\longrightarrow \RR_{\geq 0}$ satisfies
\begin{align}\label{as:range-a}
	a\in C^{\alpha,\frac\alpha2}(\Omega_T,\RR_{\geq 0})\quad\text{for some}\quad\alpha\in(0,1] \quad\text{and}\quad 
    q\le p+\alpha. 
\end{align}
Here $a\in C^{\alpha,\frac\alpha2}(\Omega_T)$ means that $a\in L^{\infty}(\Omega_T)$ and that there exists a constant $[a]_{\alpha}<\infty$ such that
\begin{align*}
	\begin{split}
		|a(x,t)-a(y,t)|\le [a]_{\alpha}|x-y|^\alpha\quad\text{and}\quad
		|a(x,t)-a(x,s)|\le [a]_{\alpha}|t-s|^\frac{\alpha}{2}
	\end{split}
\end{align*}
for every $(x,y)\in\Omega$ and $(t,s)\in (0,T)$. 

Next, we recall the definition of a weak solution to~\eqref{double-phase} 

\begin{definition} \label{def:weak-sol}
	A measurable function $u:\Om_T\longrightarrow\RR$ with
	$$
 u\in C(0,T;L^2(\Om))\cap L^q(0,T;W^{1,q}(\Om))
    $$
is a weak solution to \eqref{double-phase} if
$$
\iint_{\Om_T}(-u \partial_t \varphi +\mA(x,t,u,\na u)\cdot \na \varphi)\, \d x \d t=0
$$
for every $\varphi\in C_0^\infty(\Omega_T)$.
\end{definition}

In particular, this paper considers a weak solution to \eqref{double-phase} with the boundedness assumption, that is,
$$
\|u\|_\infty=\|u\|_{L^\infty(\Om_T)}<\infty.
$$
For simplicity, we denote the constant dependency $c(\data)$ if a constant $c>0$ depends (at most) on
$$
n,p,q,\alpha,C_0,C_1,[a]_\alpha,\|u\|_\infty.
$$
Furthermore, instead of $\essinf$, $\esssup$ and $\essosc$ we simply write $\inf$, $\sup$ and $\osc$.

\subsection{Auxiliary results}

In this section we state some standard results that are useful in the proof. First, we state a lemma on fast geometric convergence and DeGiorgi's isoperimetric inequality~\cite[Chapter I, Lemma 4.1 \& Lemma 2.2]{DB-book}.

\begin{lemma} \label{lem:fgc}
Suppose that $\{Y_i\}_{i\in \mathbb{N}_0}$ is a sequence of positive real numbers that satisfy
$$
Y_{i+1} \leq CB^i Y_i^{1+ \sigma}\quad \text{for all}\quad i \geq 0,
$$
with constants $C,\sigma > 0$ and $B>1$. Then $Y_i \to 0$ as $i \to \infty$ whenever
$$
Y_0 \leq C^{- \frac{1}{\sigma}} B^{- \frac{1}{\sigma^2}}.
$$
\end{lemma}

\begin{lemma} \label{lem:isoperimetric}
Let $k<l$ be real numbers and $B_\rho(x_o) \subset \RR^n$. Then for every $v\in W^{1,1}(B_\rho(x_o))$ there exists a constant $c(n) > 0$ such that 
$$
(l-k) \left| B_\rho(x_o) \cap \{ v>l \} \right| \leq \frac{c \rho^{n+1}}{\left| B_\rho(x_o) \cap \{ v<k \} \right|} \int_{B_\rho(x_o) \cap \{ k<v<l \}} \left| \nabla v \right| \, \d x.
$$
\end{lemma}

We will also exploit the following embedding, see e.g.~\cite[Chapter I, Corollary 3.1]{DB-book}.
 
\begin{definition}
    For $1<m<\infty$, let $V^m(\Om_T)$ be the function space
    $$
    V^m(\Om_T)=L^\infty(0,T;L^m(\Om))\cap L^m(0,T;W_0^{1,m}(\Om))
    $$
    with the norm
    $$
    \|v\|_{V^m(\Om_T)}=\left(\sup_{0<t<T}\int_{\Om}|v(x,t)|^m\,\d x+\iint_{\Om_T}|\na v|^m\, \d x\d t\right)^\frac{1}{m}.
    $$
\end{definition}

\begin{lemma}\label{lem:embedding_p}
    For $1<m<\infty$ and any $v\in V^m(\Om_T)$, there exists $c(n,m)$ such that
    $$
    \|v\|_{L^m(\Om_T)}^m\le c |\{\Om_T:|v|\ne 0\}|^\frac{m}{n+m}\|v\|_{V^m(\Om_T)}^m.
    $$
\end{lemma}

\section{Energy bounds} \label{sec:energy-est}
In this section we provide two energy estimates. The first estimate is the Caccioppoli type inequality.

\begin{lemma}\label{lem:caccioppoli}
	Let $u$ be a weak solution to \eqref{double-phase}. Suppose $Q_{r,s}(z_o)=B_{r}(x_o)\times (t_o-s,t_o)\subset\Omega_T$ with $r,s>0$ and $s<r^2$. Let $\varphi$ be a Lipschitz continuous function vanishing on $\partial B_r \times (t_o-s,t_o)$ with $0\le \varphi\le 1$. 
     Then for $k \in \RR$ with $|k| \leq \|u\|_\infty$, there exists a constant $c(data)$ such that
	\begin{align*}
		\begin{split}
			& \sup_{t\in (t_o-s,t_o)}\int_{B_{r}(x_o)}  (u-k)_{\pm}^2\varphi^q\, \d x \\
            &\qquad+ \iint_{Q_{r,s}(z_o)} (|\na ((u-k)_\pm\varphi^\frac{q}{p})|^p+a(z)|\na ((u-k)_\pm\varphi)|^q)\, \d x \d t  \\
			&\qquad\le c\iint_{Q_{r,s}(z_o)} ((u-k)_\pm^p (|\na \varphi|^p+r^\alpha|\na \varphi|^{q} ) +a(z_0)(u-k)_\pm^q |\na \varphi|^q )\, \d x \d t\\
			&\qquad\qquad+c\iint_{Q_{r,s}(z_o)} (u-k)_\pm^2 |\partial_t \varphi^q| \, \d x \d t + c\int_{B_r(x_o) \times \{t_o-s\}}  (u-k)_\pm^2\varphi^q \, \d x.
		\end{split}
	\end{align*}
\end{lemma}

\begin{proof}
 
By using the standard mollifier with parameter $\varepsilon$,
 
the mollified weak formulation can be written as
$$
\iint_{\Omega_T} \partial_t u_\eps  \varphi + (\mA(x,t,u,\nabla u))_\eps \cdot \nabla \varphi \, \d x \d t =\, 0.
$$
Suppose that $\varphi$ is given as in the assumptions and 
 
let $\eta_h (t)$ be a piecewise affine approximation of the characteristic function $\eta(t)=\chi_{(t_1,t_2)}$ with parameter $h$, in which $t_o-s < t_1 < t_2 < t_o$. 
 
Then, by testing with $\varphi^q(x,t) \eta_h(t) (u_\eps - k)_+$ we obtain for the parabolic part
$$
\iint_{\Omega_T} \partial_t u_\eps \varphi^q \eta_{ h } (u_\eps - k)_+ \, \d x \d t
= - \frac12 \iint_{\Omega_T} (\eta_{ h } \partial_t \varphi^q + \pa_t\eta_{ h } \varphi^q ) (u_\eps - k)_+^2 \, \d x \d t.    
$$
Consequently, we get
\begin{align*}
    \begin{split}
      &\iint_{\Omega_T} \partial_t u_\eps \varphi^q \eta_{ h } (u_\eps - k)_+ \, \d x \d t\\
       & \qquad\xrightarrow{\eps \to 0} -\frac12 \iint_{\Omega_T} \eta_{ h } \partial_t \varphi^q  (u - k)_+^2 \, \d x \d t- \frac12 \iint_{\Omega_T} \pa_t\eta_{ h } \varphi^q  (u - k)_+^2 \, \d x \d t\\
       &\qquad= \mathrm{I} + \mathrm{II}.
    \end{split}
\end{align*}
Recalling $\eta_{ h }$ is a approximation of $\chi_{(t_1,t_2)}$ 
 
and $u\in C(0,T;L^2(\Om))$,
 
there holds
$$
\mathrm{II} \xrightarrow{h \to 0} - \frac12 \int_{\Omega \times \{t_1\}} \varphi^q (u - k)_+^2 \, \d x + \frac12 \int_{\Omega \times \{t_2\}} \varphi^q (u - k)_+^2 \, \d x.
$$
For the divergence part we obtain
\begin{align*}
\begin{split}
&\iint_{\Omega_T} \eta _{ h }(\mA(x,t,u,\nabla u))_\eps \cdot \nabla (\varphi^q (u_\eps - k)_+) \, \d x \d t\\
&\qquad \xrightarrow{\eps, h  \to 0} \iint_{\Omega_T} \eta\mA(x,t,u,\nabla u) \cdot ((u - k)_+ \nabla \varphi^q  +\varphi^q \nabla (u - k)_+) \, \d x \d t.
\end{split}
\end{align*}
It follows from \eqref{as:structure} that
\begin{align*}
    \begin{split}
        &\iint_{\Omega_T} \eta \mA(x,t,u,\nabla u) \cdot ((u - k)_+ \nabla \varphi^q  +\varphi^q \nabla (u - k)_+) \, \d x \d t \\
&\qquad\geq C_0 \iint_{\Omega_T} \eta \varphi^q (|\nabla (u-k)_+|^p + a(z) |\nabla (u-k)_+|^q) \, \d x \d t \\
&\qquad\qquad - C_1 q\iint_{\Omega_T} \eta \varphi^{q-1} |\nabla \varphi|(u-k)_+ \\
&\qquad\qquad\qquad\qquad  \times(|\nabla (u-k)_+|^{p-1}+ a(z)|\nabla (u-k)_+|^{q-1}) \, \d x \d t \\
&\qquad= \mathrm{III} - \mathrm{IV}.
    \end{split}
\end{align*}
Furthermore, by Young's inequality and the fact $0\leq \varphi \leq 1$, we have
\begin{align*}
\begin{split}
    \mathrm{IV}&\leq \frac{C_0}{2} \iint_{\Omega_T} \eta \varphi^{q} (|\nabla (u-k)_+|^{p} + a(z)|\nabla (u-k)_+|^{q}) \, \d x \d t \\
&\qquad + c \iint_{\Omega_T} \eta  ( (u-k)_+^p |\nabla \varphi|^p + a(z) (u-k)_+^{q} |\nabla \varphi|^q ) \, \d x \d t
\end{split}
\end{align*}
for some $c(p,q,C_0,C_1)$.
On the other hand, since we have by \eqref{as:range-a} that $a(z)\le a(z_0)+[a]_\alpha r^\alpha$, and $(u-k)_+^q\le (3\|u\|_\infty)^{q-p}(u-k)_+^p$, there holds
\begin{align*}
\begin{split}
    &\iint_{Q_{r,s}(z_0)} a(z) (u-k)_+^{q} |\nabla \varphi|^q  \, \d x \d t \\
    &\qquad\le c\iint_{Q_{r,s}(z_0)} ((u-k)_+^p r^\alpha|\nabla \varphi|^q + a(z_0) (u-k)_+^{q} |\nabla \varphi|^q)  \, \d x \d t.
\end{split}
\end{align*}
Combining the estimates and passing $h \to 0$ yields
\begin{align*}
\begin{split}
&    \int_{B_r \times \{t_2\}} \varphi^q (u - k)_+^2 \, \d x \\
&\qquad+ \iint_{B_r \times (t_1,t_2)}  \varphi^q (|\nabla (u-k)_+|^p + a(z) |\nabla (u-k)_+|^q) \, \d x \d t \\
&\qquad\leq c \iint_{Q_{r,s}(z_0)} (u-k)_+^p (|\nabla \varphi|^p+r^\alpha|\na \varphi|^q )+ a(z_0) (u-k)_+^{q} |\nabla \varphi|^q  \, \d x \d t \\
&\quad\qquad\qquad + c \iint_{Q_{r,s}(z_0)} |\partial_t \varphi^q| (u-k)_+^2 \, \d x \d t + c \int_{B_r \times \{t_1\}} \varphi^q (u-k)_+^2 \, \d x.
\end{split}
\end{align*}

Now we may pass to the limit $t_1 \to t_o-s$. By considering separately the terms on the left hand side, we may take the supremum of $t_2$ over $(t_o-s,t_o)$ in the first term and pass to the limit $t_2 \to t_o$ in the second to conclude the claim.

The analogous calculations work if we take $-\varphi^q(x,t) \eta_h(t) (u_\eps - k)_-$ as a test function.
\end{proof}

The second estimate we introduce is the logarithmic estimate.
For constants $s,k,\vartheta$ with $0<\vartheta<s$, we denote the non-negative function
\begin{align*}
\begin{split}
\psi^{\pm}_{s,k,\vartheta}(\tau)
&=\left(\log\biggl\{ \frac{s}{(s+\vartheta)-(\tau-k)_{\pm}}\biggr\}\right)_+\\
&=
\begin{cases}
    \log\biggl\{ \frac{s}{(s+\vartheta)\pm(k-\tau)}\biggr\}&\text{if}\quad k\pm \vartheta \lessgtr \tau \lessgtr k\pm(s+\vartheta),\\
    0&\text{if}\quad \tau \gtreqqless k\pm \vartheta.
\end{cases}
\end{split}
\end{align*}
Observe that the derivative function is also non-negative and defined as
$$
    (\psi^{\pm}_{s,k,\vartheta})'(\tau)=
    \begin{cases}
        \frac{1}{k-\tau\pm(s+\vartheta)}&\text{if}\quad k\pm\vartheta \lessgtr\tau \lessgtr k\pm(s+\vartheta),\\
        0&\text{if}\quad \tau\lessgtr k\pm\vartheta.
    \end{cases}
$$
Moreover,  we also have
\begin{align}\label{log_second}
    (\psi^{\pm}_{s,k,\vartheta})''(\tau)=\bigl((\psi^{\pm}_{s,k,\vartheta})'(\tau)\bigr)^2.
\end{align}
To simplify the notation, we write
$$
    s^{\pm}=s^{\pm}(k)=\sup_{Q_{r,s}(z_0)}(u-k)_{\pm}.
$$
and 
\begin{align}\label{def_psi}
    \psi^{\pm}(u)= \psi^{\pm}_{s^\pm,k,\vartheta}\circ u.
\end{align}

\begin{lemma}\label{lem:logarithmic}
 Let $u$ be a weak solution to \eqref{double-phase}. Suppose $Q_{r,s}(z_o)=B_{r}(x_o)\times (t_o-s,t_o)\subset\Omega_T$ with $r,s>0$ and $s<r^2$. Let $\varphi$ be a Lipschitz continuous spatial function vanishing on $\partial B_r$ with $0\le \varphi\le 1$. 
     Then for $k \in \RR$ with $|k| \leq \|u\|_\infty$, there exists a constant $c(data)$ such that
 \begin{align*}
  \begin{split}
			& \sup_{t\in (t_o-s,t_o)}\int_{B_{r}(x_o)} \varphi^q |\psi^{\pm}(u)|^2\, \d x \\
			&\qquad\le c\iint_{Q_{r,s}(z_o)} \psi^{\pm}(u)\bigl((\psi^\pm)'(u)\bigr)^{2-p}(|\na\varphi|^p+r^\alpha|\na\varphi|^q)\, \d x \d t\\
            &\qquad\qquad+c\iint_{Q_{r,s}(z_o)}a(z_0)\psi^{\pm}(u)\bigl((\psi^\pm)'(u)\bigr)^{2-q}|\na\varphi|^q \, \d x \d t\\
			&\qquad\qquad+\int_{B_r(x_o) \times \{t_o-s\}} \varphi^q |\psi^{\pm}(u)|^2 \, \d x.
		\end{split}
	\end{align*}
\end{lemma}
\begin{proof}
    As in the proof of Lemma~\ref{lem:caccioppoli}, we take $2\eta_h\varphi^q\psi^\pm(u_\ep)(\psi^\pm)'(u_\ep)$ as a test function to the mollified weak formulation. Then applying integration by parts to the time derivative part, we get
    \begin{align*}
        \begin{split}
            &\iint_{\Om_T}\pa_tu_\ep2\eta_{ h }\varphi^q\psi^\pm(u_\ep)(\psi^\pm)'(u_\ep)\,\d x\d t\\
            &\qquad=\iint_{\Om_T}\varphi^q\eta_{ h }\pa_t |\psi^\pm(u_\ep)|^2\,\d x\d t=-\iint_{Q_{r.s}(z_0)}\pa_t\eta_{ h }\varphi^q|\psi^\pm(u_\ep)|^2\,\d x\d t
        \end{split}
    \end{align*}
    By passing $\epsilon$ and $h$ to $0$ and using 
$u\in C(0,T;L^2(\Om))$, the time derivative part converges to 
    $$
    \int_{B_r(x_0)\times \{t_2\}}\varphi^q|\psi^\pm(u)|^2\,\d x-\int_{B_r(x_0)\times\{t_1\}}\varphi^q|\psi^\pm(u)|^2\,\d x.
    $$
    On the other hand, in order to estimate the spatial derivative part we first estimate the $p$-Laplace part of the double-phase operator , that is    
    \begin{align*}
        \begin{split}
        &\iint_{\Om_T}|\na u|^{p-2}\na u\cdot \na \bigl( 2\eta\varphi^q\psi^\pm(u)(\psi^\pm)'(u) \bigr) \,\d x\d t\\
        &\qquad=\iint_{\Om_T}\eta|\na u|^{p-2}\na u\cdot \na u\bigl( 2\bigl((\psi^{\pm})'(u)\bigr)^2+2\psi^\pm(u)(\psi^\pm)^{''}(u) \bigr)\varphi^q \,\d x\d t\\
        &\qquad\qquad+\iint_{\Om_T}\eta|\na u|^{p-2}\na u\cdot \na \varphi\bigl( 2q \varphi^{q-1}\psi^\pm(u)(\psi^\pm)'(u) \bigr) \,\d x\d t\\
        &\qquad=\mathrm{I}+\mathrm{II}.
        \end{split}
    \end{align*}
Applying \eqref{log_second}, we have
$$
    \mathrm{I}= \iint_{\Om_T}2\eta|\na u|^p\bigl(1+\psi^\pm(u)\bigr)\bigl((\psi^{\pm})'(u)\bigr)^2\varphi^q\,\d x\d t.
$$
    To estimate $\mathrm{II}$, we use Young's inequality to have
    \begin{align*}
        \begin{split}
            \mathrm{II}
            &\ge-\iint_{\Om_T}2\Bigl(\eta|\na u|^{p}\psi^\pm(u)\bigl((\psi^\pm)'(u)\bigr)^2\varphi^q \Bigr)^\frac{p-1}{p} \\
            &\qquad\qquad\times
            q\eta^\frac{1}{p}(\psi^\pm(u))^\frac{1}{p}\bigl((\psi^\pm)'(u)\bigr)^{1-\frac{2(p-1)}{p}}|\na \varphi|\varphi^\frac{q-p}{p} \,\d x\d t\\
            &\ge -\iint_{\Om_T}2\eta|\na u|^{p}\psi^\pm(u)\bigl((\psi^\pm)'(u)\bigr)^2\varphi^q\,\d x\d t\\
            &\qquad-c\iint_{\Om_T}\eta \psi^\pm(u)\bigl((\psi^\pm)'(u)\bigr)^{2-p}|\na\varphi|^p\varphi^{q-p}\,\d x\d t
        \end{split}
    \end{align*}
    for some $c(p,q)$. Since the first term on the right hand side cancels the term in $\mathrm{I}$, we obtain
    \begin{align*}
        \begin{split}
            &\iint_{\Om_T}|\na u|^{p-2}\na u\cdot \na \bigl( 2\eta\varphi^q\psi^\pm(u)(\psi^\pm)'(u) \bigr) \,\d x\d t\\
            &\qquad\ge \iint_{\Om_T}2\eta|\na u|^p\bigl((\psi^{\pm})'(u)\bigr)^2\varphi^q\,\d x\d t\\
            &\qquad\qquad-c\iint_{\Om_T}\eta \psi^\pm(u)\bigl((\psi^\pm)'(u)\bigr)^{2-p}|\na\varphi|^p\,\d x\d t.
        \end{split}
    \end{align*}
    
    Similarly, we have the following estimate for the $q$-Laplace part of the double-phase equation:
    \begin{align*}
        \begin{split}
            &\iint_{\Om_T}a|\na u|^{q-2}\na u\cdot \na \bigl( 2\eta\varphi^q\psi^\pm(u)(\psi^\pm)'(u) \bigr) \,\d x\d t\\
            &\qquad\ge \iint_{\Om_T}2\eta a|\na u|^q\bigl((\psi^{\pm})'(u)\bigr)^2\varphi^q\,\d x \d t\\
            &\qquad\qquad-c\iint_{\Om_T}\eta a\psi^\pm(u)\bigl((\psi^\pm)'(u)\bigr)^{2-q}|\na\varphi|^q\,\d x\d t.
        \end{split}
    \end{align*}
    Moreover, since $\varphi$ is supported in $Q_{r,s}(z_0)$ with $s<r^2$, it follows from \eqref{as:range-a} that in the last term we may replace $a(z)$ by $a(z_0)+[a]_\alpha r^\alpha$. On the other hand, since $|k|\le \|u\|_\infty$ gives $\vartheta<s^\pm\le 3\|u\|_\infty$, we get $\bigl((\psi^\pm)'(u)\bigr)^{p-q}\le c\|u\|_\infty^{q-p}$ and
    \begin{align*}
        \begin{split}
            &\iint_{\Om_T}\eta a\psi^\pm(u)\bigl((\psi^\pm)'(u)\bigr)^{2-q}|\na\varphi|^q\,\d x\d t\\
            &\qquad\le c\iint_{\Om_T}\eta a(z_0)\psi^\pm(u)\bigl((\psi^\pm)'(u)\bigr)^{2-q}|\na\varphi|^q\,\d x\d t\\
            &\qquad\qquad+c\iint_{\Om_T}\eta \psi^\pm(u)\bigl((\psi^\pm)'(u)\bigr)^{2-p}r^\alpha|\na\varphi|^q\,\d x\d t.
        \end{split}
    \end{align*}
    Combining all the estimates and passing $t_1 \to t_o-s$ while taking $t_2$ be arbitrary, the conclusion follows.
\end{proof}

\section{Reduction of oscillation} \label{sec:reduction-osc}

For simplicity, we denote $(x_o,t_o) = (0,0)$ and   $Q_{r,s}(0,0)=Q_{r,s}$. L et $\rho \in (0,1)$ such that $Q_{\rho,\rho^2} \subset \Omega_T$ and $\bomega \in (0,2 \|u\|_\infty]$. For $A\ge 4$ to be determined later in Lemma~\ref{lem:A}, we further assume
\begin{align}\label{62}
    A^\frac{1}{p-2}\rho<\bomega.
\end{align}
By denoting $a_0=a(0,0)$, we write the scaling factor of the double-phase equation at $(0,0)$ as
$$
    \theta_r = \theta_{(\bomega, r)}=\Bigl(\frac{\bomega}{r}\Bigr)^2\Bigl(\Bigl(\frac{\bomega}{r}\Bigr)^p+a_0\Bigl(\frac{\bomega}{r}\Bigr)^q  \Bigr)^{-1}.
$$
Note that $A\theta_\rho\le 1$ since $A\theta_\rho\le A\bigl(\tfrac{\bomega}{\rho}\bigr)^{2-p}$  and \eqref{62} hold. We suppose that $\mathcal{Q}$ is a cylinder such that $Q_{\rho,A \theta_{(\bomega,\rho)} \rho^2} \subset \mathcal{Q} \subset Q_{\rho,\rho^2} \subset \Omega_T$, let
$$
    \muplus = \sup_{\mathcal Q}u\quad\text{and}\quad\muminus= \inf_{\mathcal Q}u
$$
and suppose that $\muplus - \muminus \leq \bomega$ is satisfied together with 
\begin{equation} \label{eq:omega-alternative`}
\muplus-\muminus > \tfrac12 \bomega.
\end{equation}

In this section we will analyze the phase of \eqref{double-phase} by comparing $\bigl( \tfrac{\bomega}{\rho}\bigr)^p$ with $a_0\bigl( \tfrac{\bomega}{\rho}\bigr)^q$. In fact, these two will appear in our arguments by estimating the right-hand side of the Caccioppoli inequality in Lemma~\ref{lem:caccioppoli} further and the phase analysis will play a role in determining an intrinsic geometry so that the embedding in Lemma~\ref{lem:embedding_p} can be utilized. In the following lemma, we consider   the case   when the term with exponent $q$ is larger than the term with exponent $p$ which corresponds to the $(p,q)$-phase in \cite{MR4627284,KS}.
\begin{lemma}\label{phase_criterion}
    Let $K=\max \{1, 4[a]_\alpha\|u\|_\infty^{q-p}\}$. If $K\bigl( \tfrac{\bomega}{\rho}\bigr)^p<a_0\bigl( \tfrac{\bomega}{\rho}\bigr)^q$ holds, then we have
    $$
        \frac{a_0}{2}\le a(z)\le 2a_0\quad\text{for all}\quad z\in Q_{\rho,\rho^2}.
    $$
\end{lemma}

\begin{proof}
  We claim that
  $$
      \inf_{z\in Q_{\rho,\rho^2}}a(z)\ge [a]_\alpha\rho^\alpha.
  $$
      Otherwise, it follows from the assumption that
      $$
          K\Bigl(\frac{\bomega}{\rho}\Bigr)^p< 2[a]_\alpha\rho^\alpha\Bigl(\frac{\bomega}{\rho}\Bigr)^q.
      $$
      Here we used \eqref{as:range-a} to have
      $$
          a_0 \leq \inf_{z\in Q_{\rho,\rho^2}} a(z) + [a]_\alpha \rho^\alpha < 2 [a]_\alpha \rho^\alpha.
      $$
      It follows that $K< 2[a]_\alpha \rho^{p+\alpha-q}\bomega^{q-p}$. As $4[a]_\alpha\|u\|_\infty^{q-p}\le K$, $\rho\in(0,1)$ and $\bomega\le 2\|u\|_\infty$, we have $4[a]_\alpha\|u\|_\infty^{q-p}< 2^{1+q-p}[a]_\alpha \|u\|_\infty^{q-p}$ which is a contradiction. Therefore, the claim is true and from it we deduce that
      $$
          \sup_{z\in Q_{\rho,\rho^2}}a(z)\le \inf_{z\in Q_{\rho,\rho^2}}a(z)+[a]_\alpha\rho^\alpha \le 2\inf_{z\in Q_{\rho,\rho^2}}a(z).
      $$
      The conclusion of the lemma follows from the above inequality.
\end{proof}

In the rest of the paper we suppose that
\begin{equation} \label{eq:K} 
K=\max \{1, 4[a]_\alpha\|u\|_\infty^{q-p}\}.
\end{equation}

\subsection{First Alternative}
Let $\nu_0 \in (0,1)$. We say that the first alternative holds if there exists some $t^*\in(-(A-1)\theta_\rho\rho^2,0)$ such that
\begin{equation} \label{eq:first-alt`}
\left| Q_{\rho,\theta_\rho\rho^2}(0,t^*) \cap  \left\{ u \leq \muminus + \tfrac{\bomega}{4}  \right\} \right| \leq \nu_0 \left| Q_{\rho,\theta_\rho \rho^2}\right|.
\end{equation}
Next, we state the DeGiorgi type lemma. 

\begin{lemma} \label{lem:degiorgi-first-alt`}
There exists $\nu_0 (\data)  \in (0,1)$ such that if \eqref{eq:first-alt`} holds,
$$
u > \muminus + \frac{\bomega}{8} \quad \text{a.e. in}\quad Q_{\rho/(4K),\theta_{\rho/(4K)} \left(\rho/(4K)\right)^2} (0,t^*).
$$
\end{lemma}

\begin{proof}
    For the proof we consider the following cases:
    \begin{enumerate}[label=(\arabic*)]
        \item\label{q} $a_0\bigl(\tfrac{\bomega}{\rho}\bigr)^q>K\bigl(\tfrac{\bomega}{\rho}\bigr)^p$ holds.
        \item\label{p} $a_0\bigl(\tfrac{\bomega}{\rho}\bigr)^q\le K\bigl(\tfrac{\bomega}{\rho}\bigr)^p$ holds.
    \end{enumerate}
    \textit{Case}~\ref{q}. In this case, we observe that the term with exponent $q$ dominates the scaling factor and the time interval is estimated as in
    $$
        \Bigl(\frac{\bomega}{\rho}\Bigr)^2\Bigl(a_0\Bigl(\frac{\bomega}{\rho}\Bigr)^q\Bigr)^{-1}\rho^2>\theta_\rho\rho^2>\Bigl(\frac{\bomega}{\rho/2}\Bigr)^2\Bigr(a_0\Bigl(\frac{\bomega}{\rho/2}\Bigr)^q\Bigl)^{-1}\left(\rho/2\right)^2.
    $$
    Using the above estimate, \eqref{eq:first-alt`} becomes
    \begin{align}\label{613}
       \bigl|Q_{\rho/2,a_0^{-1}\bomega^{2-q}(\rho/2)^q}(0,t^*)\cap \left\{u\le  \muminus+\tfrac{\bomega}{4}\right\}\bigr|\le 2^{n+q}\nu_0\bigl|Q_{\rho/2,a_0^{-1}\bomega^{2-q}(\rho/2)^q}\bigr|.
    \end{align}

For each $j\in\mathbb{N}$, define
\begin{equation}\label{614}
    \begin{aligned}[c]
        \rho_j &= \frac{\rho}{4} + \frac{\rho}{2^{j+2}},\\
        Q_j &= B_j\times I_j,
    \end{aligned}
    \qquad
    \begin{aligned}[c]
        B_j &= B_{\rho_j},\\
        I_j&=(t^*-a_0^{-1}\bomega^{2-q}\rho_j^q,t^*).
    \end{aligned}
\end{equation}
Let $0 \leq \varphi_j\leq 1$ be a Lipschitz function such that $\varphi_j = 1$ in $Q_{j+1}$ and $\varphi_j$ vanishes on the parabolic boundary of $Q_j$ and 
\begin{align}\label{311}
    |\nabla \varphi_j| \leq c\frac{2^j}{\rho} \quad \text{and}\quad |\partial_t \varphi_j| \leq c \frac{2^{qj}a_0\bomega^{q-2}}{ \rho^q}.
\end{align}
Finally, let
\begin{align}\label{615}
    k_j = \muminus + \frac{\bomega}{8} + \frac{\bomega}{2^{j+3}}.
\end{align}
Since $Q_0\subset Q_{\rho,\rho^2}$, it follows from Lemma~\ref{lem:caccioppoli}
and Lemma~\ref{phase_criterion} that
    \begin{align}\label{312}
		\begin{split}
			& \sup_{t\in I_0}\int_{B_0} (u-k_j)_-^2\varphi_j^q\, \d x + \iint_{Q_0} a_0|\na ((u-k_j)_-\varphi_j)|^q\, \d x \d t  \\
			&\qquad\le c\iint_{Q_0} (u-k_j)_-^p (|\na \varphi_j|^p +\rho^\alpha|\na\varphi_j|^q)+a_0(u-k_j)_-^q |\na \varphi_j|^q \, \d x \d t\\
			&\qquad\qquad+c\iint_{Q_0} (u-k_j)_-^2 |\partial_t \varphi_j^q| \, \d x \d t
		\end{split}
	\end{align}
for some $c(\data)$. 
Note that the definitions of $k_j$ and $\muminus$ imply 
\begin{align}\label{313}
    (u-k_j)_-\le \frac{\bomega}{4},\quad \Bigl(\frac{\bomega}{4}\Bigr)^{2-q}(u-k_j)_-^q\le (u-k_j)_-^2\quad\text{in}\quad\mathcal{Q}.
\end{align}
Using \eqref{311}, \eqref{313} and \ref{q}, the Caccioppoli inequality in \eqref{312} becomes
        \begin{align}\label{314}
		\begin{split}
			& a_0^{-1}\left(\frac{\bomega}{4}\right)^{2-q}\sup_{t\in I_0}\int_{B_0} ((u-k_j)_-\varphi_j)^q\, \d x + \iint_{Q_0} |\na ((u-k_j)_-\varphi_j)|^q\, \d x \d t  \\
			&\qquad\le c \frac{2^{qj}}{\rho^q}\left(\frac{\bomega}{4}\right)^{q}\iint_{Q_0}\chi_{\{Q_j:u < k_j\}}    \, \d x \d t.
		\end{split}
	\end{align}  
Next, we do a change of variables with $\bar u (\cdot, t^*+\bar t) = u(\cdot,t^*+t)$ and $\bar \varphi_j (\cdot,t^*+\bar t) = \varphi_j (\cdot,t^*+t)$, where $\bar t = t/(a_0^{-1}\bomega^{2-q})$, and denote $\bar Q_j=B_j\times (t^*-\rho_j^q,t^*)$ and $A_{j} = \bar Q_j \cap \left\{ \bar u < k_j \right\}$. Then \eqref{314} becomes
\begin{align}\label{315}
		\begin{split}
			& \sup_{\bar t\in (t^*-\rho_j^q,t^*)}\int_{B_j} ((\bar u-k_j)_-\bar\varphi_j)^q\, \d x + \iint_{\bar Q_j} |\na ((\bar u-k)_-\bar \varphi_j)|^q\, \d x \d t  \\
			&\qquad\le c \frac{2^{qj}}{\rho^q}\left(\frac{\bomega}{4}\right)^{q}\iint_{\bar Q_j}\chi_{A_j}    \, \d x \d t,
		\end{split}
	\end{align}
 where we replaced
 $\bar{Q}_0$ by $\bar{Q}_{j}$ since $\bar{\varphi}_j$ is supported there. As $k_j-k_{j+1}=\tfrac{\bomega}{2^{j+4}}$, we have
 \begin{align*}
     \begin{split}
     2^{-q(j+4)} \bomega^q |A_{j+1}| &= \iint_{\bar Q_{j+1}}(k_j-k_{j+1})^q \chi_{\{\bar u < k_{j+1}\}} \, \d x \d t\\
     & \leq \iint_{\bar Q_{j+1}}(k_j-\bar u)^q \chi_{\{\bar u < k_j\}} \, \d x \d t\\
     & = \|(\bar u-k_j)_- \bar \varphi_j\|_{L^q(\bar Q_j)}^q.
     \end{split}
 \end{align*}
Therefore, applying Lemma~\ref{lem:embedding_p} and \eqref{315} gives
 \begin{align}\label{316}
 \begin{split}
     2^{-q(j+4)} \bomega^q |A_{j+1}| &\leq \|(\bar u-k_j)_- \bar \varphi_j\|_{L^q(\bar Q_j)}^q \\
&\leq c \|(\bar u-k_j)_- \bar \varphi_j\|_{V^q(\bar Q_j)}^q |A_j|^\frac{q}{n+q} \\
&\leq c \frac{2^{qj}}{\rho^q} \left( \frac{\bomega}{4} \right)^q |A_j|^{1+\frac{q}{n+q}}.
 \end{split}
\end{align}
By dividing the inequality above by $|\bar Q_{j+1}|$ and denoting $Y_j = |A_j| / |\bar Q_j|$, we have
$$
Y_{j+1} \leq c 2^{2qj} Y_j^{1 + \frac{q}{n+q}}.
$$
Recall that $Y_0\le 2^{n+q}\nu_0$ holds by \eqref{613}. By choosing $C = c$, $B = 2^{2q}$ and $\sigma = \tfrac{q}{n+q}$ in Lemma~\ref{lem:fgc}, we have
$$
    u(x,t)>\muminus+\frac{\bomega}{8}\quad\text{a.e. in}\quad Q_{\rho/4,a_0^{-1}\bomega^{2-q}(\rho/4)^q}(0,t^*),
$$
provided that
\begin{align}\label{617}
    2^{n+q}\nu_0\le c^{- \frac{n+q}{q}} 2^{-2q \frac{(n+q)^2}{q^2}}.
\end{align}

\textit{Case}~\ref{p}. The proof is analogous to the previous case. The only difference is that we do the estimates in $p$-intrinsic cylinders. Indeed, by \ref{p} we have
$$
        \Bigl(\frac{\bomega}{\rho}\Bigr)^{2-p}\rho^2>\theta_\rho\rho^2>\biggl(\frac{\bomega}{\rho/(2K)}\biggr)^{2-p}\left(\rho/(2K)\right)^2,
$$
and \eqref{eq:first-alt`} can be written as
\begin{align*}
\begin{split}
    &\left|Q_{\rho/(2K),\bomega^{2-p}(\rho/(2K))^p}(0,t^*) \cap  \left\{ u \leq \muminus + \tfrac{\bomega}{4}  \right\} \right| \\
    &\qquad\leq (2K)^{n+p}\nu_0 \left| Q_{\rho/(2K),\bomega^{2-p}(\rho/(2K))^p}\right|. 
\end{split}
\end{align*}

For the index $j\in\mathbb{N}$, we replace \eqref{614} by 
$$
    \rho_j = \frac{\rho}{4K} + \frac{\rho}{2^{j+2}K},\quad B_j = B_{\rho_j},\quad I_j=(t^*-\bomega^{2-p}\rho_j^p,t^*),\quad Q_j = B_j\times I_j
$$
and take a Lipschitz function $0 \leq \varphi_j\leq 1$ such that $\varphi_j = 1$ in $Q_{j+1}$ and $\varphi=0$ on   the parabolic boundary of $Q_j$ , with
\begin{align}\label{319}
    |\nabla \varphi_j| \leq c\frac{2^j}{\rho} \quad \text{and}\quad |\partial_t \varphi_j| \leq c \frac{2^{pj }\bomega^{p-2}}{ \rho^p},
\end{align}
where $c= c(\data)$. Again, for $k_j$ defined in \eqref{615} we have
    \begin{align*}
		\begin{split}
			& \sup_{t\in I_0}\int_{B_0} (u-k_j)_-^2\varphi_j^q\, \d x + \iint_{Q_0} |\na ((u-k_j)_-\varphi_j^\frac{q}{p})|^p\, \d x \d t  \\
			&\qquad\le c\iint_{Q_0} (u-k_j)_-^p (|\na \varphi_j|^p+\rho^\alpha|\na\varphi_j|) +a_0(u-k_j)_-^q |\na \varphi_j|^q \, \d x \d t\\
			&\qquad\qquad+c\iint_{Q_0} (u-k_j)_-^2 |\partial_t \varphi_j^q| \, \d x \d t
		\end{split}
	\end{align*}
for some $c(\data)$. Moreover, using 
$$
  (u-k_j)_-\le \frac{\bomega}{4},\quad  \Bigl(\frac{\bomega}{4}\Bigr)^{2-p}(u-k_j)_-^p\le (u-k_j)_-^2
$$
along with \ref{p} and \eqref{319}, we get
\begin{align*}
		\begin{split}
			& \left(\frac{\bomega}{4}\right)^{2-p}\sup_{t\in I_0}\int_{B_0} ((u-k_j)_-\varphi_j^\frac{q}{p})^p\, \d x + \iint_{Q_0} |\na ((u-k_j)_-\varphi_j^\frac{q}{p})|^p\, \d x \d t  \\
			&\qquad\le c \frac{2^{qj}}{\rho^p}\left(\frac{\bomega}{4}\right)^{p}\iint_{Q_0}\chi_{\{Q_j:u < k_j\}}    \, \d x \d t.
		\end{split}
	\end{align*}
After a change of variables with $\bar u (\cdot, t^*+\bar t) = u(\cdot,t^*+t)$ and $\bar \varphi_j (\cdot,t^*+\bar t) = \varphi_j (\cdot,t^*+t)$, where $\bar t = t/\bomega^{2-p}$, and denoting $\bar Q_j=B_j\times (t^*-\rho_j^p,t^*)$ and $A_{j} = \bar Q_j \cap \left\{ \bar u < k_j \right\}$, we obtain 
\begin{align*}
		\begin{split}
			& \sup_{\bar t\in (t^*-\rho_j^p,t^*)}\int_{B_j} ((\bar u-k_j)_-\bar\varphi_j^\frac{q}{p})^p\, \d x + \iint_{\bar Q_j} |\na ((\bar u-k)_-\bar \varphi_j^\frac{q}{p})|^p\, \d x \d t  \\
			&\qquad\le c \frac{2^{qj}}{\rho^p}\left(\frac{\bomega}{4}\right)^{p}\iint_{\bar Q_j}\chi_{A_j}    \, \d x \d t.
		\end{split}
	\end{align*}
Again, by using the embedding theorem in Lemma~\ref{lem:embedding_p} we get
\begin{align*}
    \begin{split}
 2^{-p(j+4)} \bomega^p |A_{j+1}| &\leq \|(\bar u-k_j)_- \bar \varphi_j^\frac{q}{p}\|_{L^p(\bar Q_j)}^p \\
&\leq c \|(\bar u-k_j)_- \bar \varphi_j^\frac{q}{p}\|_{V^p(\bar Q_j)}^p |A_j|^\frac{p}{n+p} \\
&\leq c \frac{2^{qj}}{\rho^p} \left( \frac{\bomega}{4} \right)^p |A_j|^{1+\frac{p}{n+p}}.
    \end{split}
\end{align*}
Therefore, $Y_{j+1} \leq c 2^{2qj} Y_j^{1 + \frac{p}{n+p}}$ holds with $Y_j = |A_j| / |\bar Q_j|$.
Recall $Y_0\le (2K)^{n+p}\nu_0$. By choosing $C = c$, $B = 2^{2q}$ and $\sigma = \tfrac{p}{n+p}$ in Lemma~\ref{lem:fgc}, we have
$$
    u(x,t)>\muminus+\frac{\bomega}{8}\quad\text{a.e. in}\quad Q_{\rho/(4K),\bomega^{2-p}(\rho/(4K))^p}(0,t^*),
$$
provided that
$$
    (2K)^{n+p}\nu_0\le c^{- \frac{n+p}{p}} 2^{-2q \frac{(n+p)^2}{p^2}}.
$$
Taking $\nu_0$ to satisfy the above inequality and \eqref{617}, the conclusion follows.
\end{proof}

\begin{remark}
    We remark on the relation between the phase criterion and the Caccioppoli inequality in the proof of the previous lemma. In case~\ref{q}, we have considered the Caccioppoli inequality in $q$-intrinsic cylinders. This and the choice of truncation level in \eqref{615} leads $\bigl(\tfrac{\bomega}{4}\bigr)^q$ to appear in the first and last terms of \eqref{316} so that $\bomega$ is canceled. The case~\ref{p} is treated similarly by taking $p$-intrinsic cylinders.
\end{remark}

In the next lemma we use the logarithmic estimate and the result of Lemma~\ref{lem:degiorgi-first-alt`} to obtain from~\eqref{eq:first-alt`} a measure density condition slice-wise up to $t=0$. We point out that the phase analysis is not necessary for it.

\begin{lemma}\label{lem:expansion_in_time}
  Suppose \eqref{eq:first-alt`} holds with $\nu_0 (\data)\in(0,1)$ chosen in Lemma~\ref{lem:degiorgi-first-alt`}. For any $\nu_1 \in (0,1)$, there exists $\bbs(\data,A,\nu_1) \in \mathbb{N}$ such that for $\bomega_{\bbs}=\bomega/\bbs$
$$
\left| B_{\rho/(8K)} \cap \left\{ u(\cdot,t) < \muminus + \bomega_{\bbs} \right\} \right| \leq \nu_1 |B_{\rho/(8K)}| \quad \text{a.e.}\quad\, t \in (-\tb,0),
$$
where 
$$
    -\tb=t^*-\theta_{\rho/(4K)}(\rho/(4K))^2.
$$
\end{lemma}

\begin{proof}
    For an integer $j\ge2$ to be determined later, we let
    $$
        k=\muminus+\frac{\bomega}{8},\quad s^-=\sup_{Q_{\rho/(4K),\tb}}(u-k)_-,\quad\vartheta=\frac{\bomega}{2^{j+3}}
    $$
    for the logarithmic function \eqref{def_psi} in the cylinder $Q_{\rho/(4K),\tb}$. 
    Note that 
    \begin{align}\label{624}
        k-u(x,t)\le s^-\le \frac{\bomega}{8}\quad\text{in}\quad Q_{\rho/(4K),\tb}
    \end{align}
    and therefore
    \begin{align}\label{625}
        \psi^-(u)=
        \begin{cases}
            \log\left\{ \frac{s^-}{s^-+\frac{\bomega}{2^{j+3}}-(k-u)} \right\}&\text{if}\quad k-\frac{\bomega}{2^{j+3}}>u,\\
            0&\text{if}\quad k-\frac{\bomega}{2^{j+3}}\le u.
        \end{cases}
    \end{align}
    Moreover, we also have
    \begin{align}\label{626}
        \psi^-(u)(x,t)\le j\log 2,\quad |(\psi^-)'(u)(x,t)|^{-1}\le s^-\quad\text{in}\quad Q_{\rho/(4K),\tb}.
    \end{align}

    Take a Lipschitz function $0\le \varphi\le 1$ on the spatial space such that $\varphi=1$ in $B_{\rho/(8K)}$, $\varphi=0$ on $\pa B_{\rho/(4K)}$ and $|\na\varphi|\le 8K/\rho$. As $Q_{\rho/(4K),\tb}\subset Q_{\rho,\rho^2}$, it follows from Lemma~\ref{lem:logarithmic} that there exists $c(\data)$ such that
     \begin{align*}
  \begin{split}
			& \sup_{t\in (-\tb,0)}\int_{B_{\rho/(4K)}} \varphi^q |\psi^-(u)|^2\, \d x \\
			&\qquad\le c\iint_{Q_{\rho/(4K),\tb}} \psi^-(u)\bigl((\psi^-)'(u)\bigr)^{2-p}(|\na\varphi|^p+\rho^\alpha|\na\varphi|^q)\, \d x \d t\\
            &\qquad\qquad +c\iint_{Q_{\rho/(4K),\tb}}a_0\psi^{-}(u)\bigl((\psi^-)'(u)\bigr)^{2-q}|\na\varphi|^q \, \d x \d t,
		\end{split}
	\end{align*}
 where we used the fact that by Lemma~\ref{lem:degiorgi-first-alt`} and \eqref{625} we have
$$
   \psi^-(u)(x,-\tb)=0\quad\text{a.e. in}\quad B_{\rho/(4K)}.
$$
Applying $|\na\varphi|\le 8K/\rho$, \eqref{624} and \eqref{626}, we get
\begin{align}\label{628}
    \begin{split}
			& \sup_{t\in (-\tb,0)}\int_{B_{\rho/(4K)}} \varphi^q |\psi^-(u)|^2\, \d x \\
			&\qquad\le cj\iint_{Q_{\rho/(4K),\tb}} \Bigl(\Bigl(\frac{\bomega}{\rho}\Bigr)^p+a_0\Bigl(\frac{\bomega}{\rho}\Bigr)^q\Bigr)\frac{1}{\bomega^2} \, \d x \d t\\
   &\qquad= cj |B_{\rho/(4K)}|\tb(\theta_\rho\rho^2)^{-1}\le c jA|B_{\rho/(4K)}|,
    \end{split}
\end{align}
 where the last inequality follows from the fact that $\tb\leq A\theta_\rho\rho^2$ by \eqref{62}.
 
To estimate the left hand side, observe that when $k-\tfrac{\bomega}{2^{j+3}}>u$, we have that $\tfrac{s^-}{s^-+\frac{\bomega}{2^{j+3}}-(k-u)}$ is a decreasing function with respect to the variable $s^-\ge k-u$. Thus, there holds
\begin{align*}
\begin{split}
    \frac{s^-}{s^-+\frac{\bomega}{2^{j+3}}-(k-u)}
    &\ge \frac{\frac{\bomega}{8}}{\frac{\bomega}{8}+\frac{\bomega}{2^{j+3}}-(k-u)} \\
    &=\frac{\frac{\bomega}{8}}{u-\muminus+\frac{\bomega}{2^{j+3}}}\ge 2^{j-1},
\end{split}
\end{align*} 
provided that $u<\muminus+\tfrac{\bomega}{2^{j+3}}$. Therefore, \eqref{628} becomes
 $$
    ((j-1)\log 2)^2 \sup_{t\in (-\tb,0)}\int_{B_{\rho/(8K)}} \chi_{\bigl\{u(x,t)<\muminus+\tfrac{\bomega}{2^{j+3}}\bigr\}}\, \d x\le c jA|B_{\rho/(8K)}|.
 $$
 The proof is completed by taking sufficiently large $j\ge2$, so that $\tfrac{cjA}{((j-1)\log 2)^2} \leq \nu_1$, and choosing $\bbs=2^{j+3}$.
\end{proof}

At this stage, we are ready to show that a reduction of oscillation is obtained in case the first alternative~\eqref{eq:first-alt`} holds. The proof is based on DeGiorgi type iteration and using Lemma~\ref{lem:expansion_in_time}, when choosing the coefficient $\tfrac{1}{B}$ of $\bomega$ small enough. In order to be able to combine Caccioppoli inequality, Lemma~\ref{lem:caccioppoli}, and the embedding in Lemma~\ref{lem:embedding_p} with the aforementioned coefficient, in the $q$-phase we consider two subcases where the phase criterion for $\bomega$ is replaced by $\bomega/B$.

\begin{lemma}
     Suppose \eqref{eq:first-alt`} holds with $\nu_0 (\data)\in(0,1)$ chosen in Lemma~\ref{lem:degiorgi-first-alt`}. Then there exists $\bb(\data,A) \in \mathbb{N}$ such that for $\bomega_{\bb}=\bomega/\bb$
$$
u > \muminus + \frac{\bomega_\bb}{8}\quad \text{ a.e. in}\quad \, Q_{\rho/(16K), \tb}.
$$
\end{lemma}

\begin{proof}
    For each $j\in\mathbb{N}$, define
    \begin{align}\label{333}
        \rho_j=\frac{\rho}{16K}+\frac{\rho}{2^{j+4}K},\quad B_j=B_{\rho_j},\quad Q_j=B_j\times (-\tb,0).
    \end{align}
    Let $0\le\varphi_j\le 1$ be a Lipschitz function on the spatial variables such that $\varphi_j=1$ in $B_{j+1}$ and $\varphi=0$ on $\pa B_j$ with
    \begin{align}\label{334}
        |\na\varphi_j|\le c\frac{2^j}{\rho},
    \end{align}
    where $c=c(\data).$ Also, for $\bb\geq2$ chosen later, consider
    $$
        k_j=\muminus+\frac{\bomega_\bb}{8}+\frac{\bomega_\bb}{2^{j+3}}.
    $$
    Note that it follows from Lemma~\ref{lem:degiorgi-first-alt`} that 
    $$
        (u-k_j)_-(x,-\tb)=0\quad\text{in}\quad B_{\rho/(4K)}.
    $$
    
    In order to prove the lemma, we will follow the analogous argument in the proof of Lemma~\ref{lem:degiorgi-first-alt`}. Again, it follows from the definitions of $k_j$ and $\muminus$ that 
    \begin{align}\label{337}
        (u-k_j)_-\le \frac{\bomega_\bb}{4} \quad\text{in}\quad \mathcal{Q}.
    \end{align}
    After using the above estimate, terms involving $\bomega_\bb^p$ and $a_0\bomega_\bb^q$ will appear on the right hand side of the Caccioppoli inequality. Thus, we consider the following cases:

    \begin{enumerate}[label=(\alph*)]
        \item\label{1} $a_0\bigl(\tfrac{\bomega}{\rho}\bigr)^q>K\bigl(\tfrac{\bomega}{\rho}\bigr)^p$ and $a_0\bigl(\tfrac{\bomega_\bb}{\rho}\bigr)^q>K\bigl(\tfrac{\bomega_\bb}{\rho}\bigr)^p$ hold.
        \item\label{2} $a_0\bigl(\tfrac{\bomega}{\rho}\bigr)^q>K\bigl(\tfrac{\bomega}{\rho}\bigr)^p$ and $a_0\bigl(\tfrac{\bomega_\bb}{\rho}\bigr)^q\le K\bigl(\tfrac{\bomega_\bb}{\rho}\bigr)^p$ hold.
        \item\label{3} $a_0\bigl(\tfrac{\bomega}{\rho}\bigr)^q\le K\bigl(\tfrac{\bomega}{\rho}\bigr)^p$ holds.
    \end{enumerate}
    \textit{Case}~\ref{1}. In this case, we apply Lemma~\ref{lem:caccioppoli} with \eqref{333}-\eqref{337}, Lemma~\ref{phase_criterion} and \ref{1}, to obtain
    \begin{align*}
		\begin{split}
			& \sup_{t\in (-\tb,0)}\int_{B_0} (u-k_j)_-^2\varphi_j^q\, \d x + \iint_{Q_0} a_0|\na ((u-k_j)_-\varphi_j)|^q\, \d x \d t  \\
			&\qquad\le ca_0\frac{2^{qj}}{\rho^q}\Bigl(\frac{\bomega_\bb}{4}\Bigr)^q\iint_{Q_0}\chi_{\{ Q_j:u<k_j\}}\, \d x \d t
		\end{split}
	\end{align*}
for a constant $c(\data)$.
To estimate the left hand side, note that as $\tb\le A\theta_\rho\rho^2\le A \bomega^2\bigl(a_0\bigl(\tfrac{\bomega}{\rho}\bigr)^q\bigr)^{-1}$, we have
\begin{align}\label{338}
    \bomega_\bb^{2-q}\geq \frac{\bb^{q-2}a_0\tb }{A\rho^q}.
\end{align}
For $\bb$ be large enough so that $\bb^{q-2}\ge A$,
we get from \eqref{337} and \eqref{338} that
$$
        (u-k_j)_-^2\ge \Bigl(\frac{\bomega_\bb}{4}\Bigr)^{2-q}(u-k_j)_-^q\ge \frac{\bb^{q-2}}{A}\frac{\tb}{\rho^q}a_0(u-k_j)_-^q\ge \frac{\tb}{\rho^q}a_0(u-k_j)_-^q.
$$
Therefore, we have
\begin{align*}
    \begin{split}
    &\frac{\tb}{\rho^q}\sup_{t\in (-\tb,0)}\int_{B_0} (u-k_j)_-^q\varphi_j^q\, \d x + \iint_{Q_0} |\na ((u-k_j)_-\varphi_j)|^q\, \d x \d t\\
    &\qquad\le c\frac{2^{qj}}{\rho^q}\Bigl(\frac{\bomega_\bb}{4}\Bigr)^q\iint_{Q_0}\chi_{\{ Q_j:u<k_j\}}\, \d x \d t.
\end{split}
\end{align*}

Define $Y_j = |A_j| / |\bar Q_j|$, where $\bar Q_j=B_j\times (-\rho^q,0)$ and $A_{j} = \bar Q_j \cap \left\{ \bar u < k_j \right\}$, for $\bar u (\cdot, \bar t) = u(\cdot,t)$ with $\bar t = t/(\tb/\rho^q)$. As in the proof of Lemma~\ref{lem:degiorgi-first-alt`}, we obtain via an embedding argument that
$$
Y_{j+1} \leq c 2^{2qj} Y_j^{1 + \frac{q}{n+q}}.
$$
By choosing $C = c$, $B = 2^{2q}$ and $\sigma = \tfrac{q}{n+q}$ in Lemma~\ref{lem:fgc}, we have
$$
    u(x,t)>\muminus+\frac{\bomega_\bb}{8}\quad\text{a.e. in}\quad Q_{\rho/(16K),\tb},
$$
provided that
$$
    Y_0\le c^{- \frac{n+q}{q}} 2^{-2q \frac{(n+q)^2}{q^2}}.
$$
Furthermore, this condition is satisfied by using Lemma~\ref{lem:expansion_in_time} with 
\begin{align}\label{641}
    \nu_1\le c^{- \frac{n+q}{q}} 2^{-2q \frac{(n+q)^2}{q^2}}
\end{align}
and choosing $4\bb \geq \bbs$, where $\bbs(\data, A)$ comes from Lemma~\ref{lem:expansion_in_time}. This finishes the proof in the first case. 

\textit{Case}~\ref{2}-\ref{3}. Note that \ref{3} implies the second condition of \ref{2}. Therefore, in both cases Lemma~\ref{lem:caccioppoli} with \eqref{334}-\eqref{337} and the second condition of \ref{2} gives
    \begin{align*}
		\begin{split}
			& \sup_{t\in (-\tb,0)}\int_{B_0} (u-k_j)_-^2\varphi_j^q\, \d x + \iint_{Q_0} |\na ((u-k_j)_-\varphi_j^\frac{q}{p})|^p\, \d x \d t  \\
			&\qquad\le c\frac{2^{qj}}{\rho^p}\Bigl(\frac{\bomega_\bb}{4}\Bigr)^p\iint_{Q_0}\chi_{\{ Q_j:u<k_j\}}\, \d x \d t
		\end{split}
	\end{align*}
for some $c(\data)$. Using \eqref{337} and $\tb\le A \bomega^2\bigl(\tfrac{\bomega}{\rho}\bigr)^{-p}$, we get
\begin{align*}
    \begin{split}
    &\frac{\tb}{\rho^p}\sup_{t\in (-\tb,0)}\int_{B_0} (u-k_j)_-^p\varphi_j^q\, \d x + \iint_{Q_0} |\na ((u-k_j)_-\varphi_j^\frac{q}{p})|^p\, \d x \d t\\
    &\qquad\le c\frac{2^{qj}}{\rho^p}\Bigl(\frac{\bomega_\bb}{4}\Bigr)^p\iint_{Q_0}\chi_{\{Q_j: u<k_j\}}\, \d x \d t
\end{split}
\end{align*}
if $\bb$ satisfies $\bb^{p-2}\geq A$.
Again defining $\bar Q_j=B_j\times (-\rho^p,0)$, $A_{j} = \bar Q_j \cap \left\{ \bar u < k_j \right\}$ and $\bar u (\cdot, \bar t) = u(\cdot,t)$ for $\bar t = t/(\tb/\rho^p)$, we obtain
$$
Y_{j+1} \leq c 2^{2qj} Y_j^{1 + \frac{p}{n+p}}.
$$
As in the previous case, it follows from Lemma~\ref{lem:fgc} that
$$
    u(x,t)>\muminus+\frac{\bomega_\bb}{8}\quad\text{a.e. in}\quad Q_{\rho/(16K),\tb},
$$
after using Lemma~\ref{lem:expansion_in_time} with 
\begin{align}\label{646}
    \nu_1\le c^{- \frac{n+p}{p}} 2^{-2q \frac{(n+p)^2}{p^2}}
\end{align}
and taking $4\bb \geq \bbs$, where $\bbs(\data, A)$ comes from Lemma~\ref{lem:expansion_in_time}.

The proof is finished after choosing $\bb$ to satisfy  $\bb^{p-2}\geq A$ and $4\bb \geq \bbs$, where $\bbs$   is   determined by Lemma~\ref{lem:expansion_in_time} with \eqref{641} and \eqref{646}. 
\end{proof}

By combining the results of this subsection, we arrive at the reduction of oscillation when the first alternative holds.

\begin{corollary}\label{cor:alt-first-conclusion`}
    Suppose \eqref{eq:first-alt`} holds with $\nu_0 (\data)\in(0,1)$ chosen in Lemma~\ref{lem:degiorgi-first-alt`}. Then there exists $\bb(\data,A) \in \mathbb{N}$ such that
    $$
        \osc_{Q_{\lambda\rho,\theta_{\lambda\rho}(\lambda\rho)^2}}u\le \left(1-\frac{1}{8\bb}\right)\bomega,
    $$
    where $\lambda= \tfrac{1}{16K}$.
\end{corollary}

\subsection{Second Alternative}
Note that \eqref{eq:omega-alternative`} implies $\muplus - \tfrac{\bomega}{4} > \muminus + \tfrac{\bomega}{4}$.
Thus, if the first alternative \eqref{eq:first-alt`} fails, we have for any $t^*\in( -(A-1)\theta_\rho\rho^2,0)$ that
\begin{equation} \label{eq:second-alt`}
\left| Q_{\rho,\theta_\rho\rho^2}(0,t^*) \cap  \left\{ u \geq \muplus - \tfrac{\bomega}{4} \right\} \right| < (1-\nu_0) \left| Q_{\rho,\theta_\rho \rho^2}\right|.
\end{equation}

\begin{lemma}\label{lem:second-alt-slice`}
    For each $t^*\in( -(A-1)\theta_\rho\rho^2,0)$, there exists $\tc\in (t^*-\theta_\rho\rho^2,t^*-\tfrac{\nu_0}{2}\theta_\rho\rho^2)$ such that 
    \begin{align} \label{eq:t-circ-slice}
        \left|\left\{B_\rho:u(x,\tc)\ge \muplus-\tfrac{\bomega}{4}\right\}\right|\le \left( \frac{1-\nu_0}{1-\nu_0/2} \right)|B_\rho|.
    \end{align}
\end{lemma}
\begin{proof}
    For a contradiction, assume that the opposite holds, in which case
    \begin{align*}
        \begin{split}
            &\int_{t^*-\theta_\rho\rho^2}^{t^*-\tfrac{\nu_0}{2}\theta_\rho\rho^2}\left( \frac{1-\nu_0}{1-\nu_0/2} \right)|B_\rho|\, \d t\\
            &\qquad< \int_{t^*-\theta_\rho\rho^2}^{t^*-\tfrac{\nu_0}{2}\theta_\rho\rho^2}|\{B_\rho:u(x,t)\ge \muplus-\tfrac{\bomega}{4}\}|\, \d t\\
            &\qquad\le \int_{t^*-\theta_\rho\rho^2}^{t^*}|\{B_\rho:u(x,t)\ge \muplus-\tfrac{\bomega}{4}\}|\, \d t.
        \end{split}
    \end{align*}
    Applying \eqref{eq:second-alt`} to the last term, we have
    \begin{align*}
        (1-\nu_0)|Q_{\rho,\theta_\rho\rho^2}|<(1-\nu_0)|Q_{\rho,\theta_\rho,\rho^2}|,
    \end{align*}
    which is a contradiction and the proof is completed.
\end{proof}

In the following lemma we show that the measure information in~\eqref{eq:t-circ-slice} can be propagated to the whole interval $(-(A-1) \theta_\rho \rho^2,0 )$ by using the logarithmic estimate, Lemma~\ref{lem:logarithmic} and shrinking the coefficient of $\bomega$ appropriately. Observe that no phase analysis is required at this stage.

\begin{lemma}\label{lem:second-alt-expansion}
    Suppose that~\eqref{eq:second-alt`} holds true. Then there exists $\dd(\data) \in \mathbb{N}$ such that for $\bomega_\dd=\bomega/\dd$
$$
\left| B_\rho \cap \left\{ u(x,t) \ge \muplus - \bomega_\dd \right\} \right| \leq \left( 1- \left( \frac{\nu_0}{2} \right)^2 \right) |B_\rho|
$$
for a.e. $t \in \left(-(A-1) \theta_\rho \rho^2,0\right)$.
\end{lemma}

\begin{proof}
    Fix $t^*\in (-(A-1)\theta_{\rho}\rho^2,0)$. 
    It is enough to prove that for a.e. $t \in \left(t^*-\tfrac{\nu_0}{2}\theta_\rho\rho^2,t^*\right)$ there holds
    \begin{align}\label{654}
        \left| B_\rho \cap \left\{ u(x,t) \ge \muplus - \bomega_\dd \right\} \right| \leq \left( 1- \left( \frac{\nu_0}{2} \right)^2 \right) |B_\rho|.
    \end{align}
    We set 
    \begin{align}\label{655}
        k=\muplus-\frac{\bomega}{4},\quad s^+=\sup_{Q}(u-k)_+,\quad \vartheta=\frac{\bomega}{2^{j+2}},
    \end{align}
    where $Q=B_\rho\times (\tc,t^*)$ with $\tc\in[t^*-\theta_\rho\rho^2,t^*-\tfrac{\nu_0}{2}\theta_\rho\rho^2]$ as in Lemma~\ref{lem:second-alt-slice`} and $j$ is a non-negative integer that will be determined later.
    Observe that
    \begin{align}\label{656}
        u(x,t)-k\le s^+\le \frac{\bomega}{4}\quad\text{in}\quad Q.
    \end{align}

    We consider the logarithmic function \eqref{def_psi} with \eqref{655} , that is  
    $$
        \psi^+(u)=
        \begin{cases}
            \log\left\{ \frac{s^+}{s^++\frac{\bomega}{2^{j+2}}-u+k} \right\}&\text{if}\quad k+\frac{\bomega}{2^{j+2}}<u,\\
            0&\text{if}\quad k+\frac{\bomega}{2^{j+2}}\ge u.
        \end{cases}
    $$
    It follows from \eqref{656} that
    \begin{align}\label{658}
        \psi^+(u)(x,t)\le j\log 2,\quad |(\psi^+)'(u)(x,t)|^{-1}\le s^+\quad\text{in}\quad Q.
    \end{align}
     For $\sig\in(0,1)$ determined later, let $0\le \varphi\le 1$ be a Lipschitz function on the spatial variables such that $\varphi=1$ in $B_{(1-\sigma)\rho}$, $\varphi=0$ on $\pa B_\rho$ and
     \begin{align}\label{660}
         |\na\varphi| \le \frac{1}{\sig \rho}.
     \end{align}
    Applying Lemma~\ref{lem:logarithmic} with the above $\varphi$, we get
     \begin{align*}
  \begin{split}
			& \sup_{t\in (\tc,t^*)}\int_{B_{\rho}} \varphi^q |\psi^{+}(u)|^2\, \d x \le \int_{B_\rho \times \{\tc\}} \varphi^q |\psi^{+}(u)|^2 \, \d x\\
			&\qquad +c\iint_{Q} \psi^{+}(u)\bigl((\psi^+)'(u)\bigr)^{2-p}(|\na\varphi|^p+\rho^\alpha|\na\varphi|^q)\, \d x \d t\\
            &\qquad +c\iint_{Q} a_0\psi^{+}(u)\bigl((\psi^+)'(u)\bigr)^{2-q}|\na\varphi|^q \, \d x \d t,
		\end{split}
	\end{align*}
    where $c=c(\data)$. With Lemma~\ref{lem:second-alt-slice`}, \eqref{658} and \eqref{660}, the above inequality becomes
    \begin{align*}
  \begin{split}
			& \sup_{t\in (\tc,t^*)}\int_{B_{\rho}} \varphi^q |\psi^{+}(u)|^2\, \d x \le  j^2(\log 2)^2\left(\frac{1-\nu_0}{1-\nu_0/2}\right)|B_\rho|\\
			&\qquad +c j\frac{1}{\sig^q\rho^2}\Bigl(\Bigl(\frac{\bomega}{\rho}\Bigr)^{p-2}+a_0\Bigl(\frac{\bomega}{\rho}\Bigr)^{q-2}\Bigr)|\tc-t^*||B_\rho|.
		\end{split}
	\end{align*}
   Recalling $\tc\in(t^*-\theta_\rho\rho^2,t^*-\tfrac{\nu_0}{2}\theta_\rho\rho^2)$,
   we obtain
   $$
			\sup_{t\in (\tc,t^*)}\int_{B_{(1-\sig)\rho}}|\psi^{+}(u)|^2\, \d x \le \left( j^2(\log 2)^2\left(\frac{1-\nu_0}{1-\nu_0/2}\right)+\frac{cj}{\sig^q}\right)|B_\rho|.
   $$

   In order to estimate the left hand side, we use fact that $\psi^+(u)$ is decreasing function with respect to   the   $s^+$ variable and \eqref{656} to have
   $$
       \psi^+(u)\ge \log\left\{ \frac{\frac{\bomega}{4}}{ \frac{\bomega}{4}+\frac{\bomega}{2^{j+2}}-u+k } \right\}\ge (j-1)\log 2
   $$
   provided that $\tfrac{\bomega}{4}-u+k=\muplus-u\le\tfrac{\bomega}{2^{j+2}}$.
    Thus, it follows that
    \begin{align*}
    \begin{split}
        &\sup_{t\in (\tc,t^*)} (j-1)^2\left|\left\{ B_{(1-\sigma)\rho}: u(x,t)\ge\muplus-\tfrac{\bomega}{2^{j+2}} \right\}\right|  \\
        &\qquad\le  \left( j^2\left(\frac{1-\nu_0}{1-\nu_0/2}\right)+\frac{cj}{\sig^q}\right)|B_\rho|.
    \end{split}
    \end{align*}
   As $|B_\rho\setminus B_{(1-\sigma)\rho}|\le n\sig|B_\rho|$, it follows from the above inequality that for a.e. $t\in (\tc,t^*)$
    \begin{align*}
    \begin{split}
        &\left|\left\{ B_{\rho}: u(x,t)\ge\muplus-\tfrac{\bomega}{2^{j+2}}  \right\}\right|\\
        &\qquad\le  \left( \frac{j^2(1-\nu_0)}{(j-1)^2(1-\nu_0/2)}+\frac{cj}{(j-1)^2\sig^q} +n\sig \right)|B_\rho|.
    \end{split}
    \end{align*}
    Fix $\sigma$ small enough such that $n\sig \leq \tfrac38 \nu_0^2$. Since $1<(1-\nu_0/2)(1+\nu_0)$, we may choose $j$ large enough such that $\tfrac{j^2(1-\nu_0)}{(j-1)^2(1-\nu_0/2)}<1-\nu_0^2$ and $\tfrac{cj}{(j-1)^2\sig^q}\le \tfrac{3}{8}\nu_0^2$. 
    Selecting $D=2^{j+2}$ completes the proof of \eqref{654}.
\end{proof}

In the next two lemmas we show that measure theoretic information can be transferred into pointwise information in the sense that if a solution is close to its supremum in a cylinder in a measure theoretic sense, then it is close to the supremum pointwise in a subcylinder. For $2D\le F$ and $\bomega_F$ defined to be $F^{q-2}=A$ and $\bomega_F=\bomega/F$, we consider the following cases:
        \begin{enumerate}[label=(\alph*),start=5]
        \item\label{5} $a_0\bigl(\tfrac{\bomega}{\rho}\bigr)^q>K\bigl(\tfrac{\bomega}{\rho}\bigr)^p$ and $a_0\bigl(\tfrac{\bomega_F}{\rho}\bigr)^q>K\bigl(\tfrac{\bomega_F}{\rho}\bigr)^p$ hold.
        \item\label{6} $a_0\bigl(\tfrac{\bomega}{\rho}\bigr)^q>K\bigl(\tfrac{\bomega}{\rho}\bigr)^p$ and $a_0\bigl(\tfrac{\bomega_F}{\rho}\bigr)^q\le K\bigl(\tfrac{\bomega_F}{\rho}\bigr)^p$ hold.
        \item\label{7} $a_0\bigl(\tfrac{\bomega}{\rho}\bigr)^q\le K\bigl(\tfrac{\bomega}{\rho}\bigr)^p$ holds.
    \end{enumerate}

\begin{lemma}\label{lem:degiorgi-second-alt1`}
Suppose \textnormal{\ref{5}} holds. There exists $\nu_2(\data) \in (0,1)$ such that if 
\begin{align}\label{666}
    \bigl| \bigl\{ u \geq \muplus - \bomega_F \bigr\} \cap Q_{\rho/2,(A/2) \theta_{\rho/2} (\rho/2)^2}  \bigr| \leq \nu_2  \bigl| Q_{\rho/2, (A/2 )\theta_{\rho/2} (\rho/2)^2} \bigr|
\end{align}
holds true, then
$$
u < \muplus-\frac{\bomega_F}{2}  \quad\text{a.e. in}\quad \, Q_{\rho/4, (A/4) \theta_{\rho/4} (\rho/4)^2}.
$$
\end{lemma}

\begin{proof}
  For each $j\in\mathbb{N}$, we take
     $$
         k_j=\muplus-\frac{\bomega_F}{2}-\frac{\bomega_F}{2^{j+1}}.
     $$
     Note that 
     \begin{align}\label{667}
         (u-k_j)_+\le \bomega_F,
     \end{align}
     and let
$$
         \rho_j=\frac{\rho}{4}+\frac{\rho}{2^{j+2}},\quad B_j=B_{\rho_j},\quad I_j=\bigl(-(4a_0)^{-1}\bomega_F^{2-q}\rho_j^q,0\bigr),\quad Q_j=B_j\times I_j.
$$
As we have
\begin{align}\label{365}
\begin{split}
    (A/4)\theta_{\rho/2}(\rho/2)^2
    &\le (A/2)(2a_0)^{-1}\bomega^{2-q}(\rho/2)^q\\
    &= (4a_0)^{-1}\bomega_F^{2-q}(\rho/2)^q\\
    &\le (A/2)\theta_{\rho/2}(\rho/2)^2,
\end{split}
\end{align}
applying the above inequality to \eqref{666} gives
\begin{align}\label{669}
    \bigl| \bigl\{ u \geq \muplus - \bomega_F \bigr\} \cap Q_0  \bigr| \leq 2\nu_2  \bigl| Q_0\bigr|.
\end{align}
Also the fact $ Q_j \subset Q_{\rho/2,(A/2)\theta_{(\rho/2)}(\rho/2)^2}$ for every $j \in \mathbb{N}$ is immediate.

Let $0\le \varphi_j\le 1$ be a Lipschitz function such that $\varphi_j=1$ in $Q_{j+1}$ and $\varphi_j=0$ on   the parabolic boundary of $Q_j$   with
$$
    |\na\varphi_j|\le c\frac{2^{j}}{\rho},\quad |\pa_t\varphi_j|\le c\frac{2^{qj} a_0\bomega_F^{q-2}}{\rho^q}.
$$
We follow the argument in the case~$\ref{q}$ of the proof of Lemma~\ref{lem:degiorgi-first-alt`}. As \eqref{667} holds, we have
\begin{align*}
    \begin{split}
    &(4a_0)^{-1}\bomega_F^{2-q}\sup_{t\in I_0}\int_{B_0} (u-k_j)_+^q\varphi_j^q\, \d x + \iint_{Q_0} |\na ((u-k_j)_+\varphi_j)|^q\, \d x \d t\\
    &\qquad\le c\frac{2^{qj}}{\rho^q}\bomega_F^q\iint_{Q_0}\chi_{\{Q_j: u>k_j\}}\, \d x \d t
\end{split}
\end{align*}
where $c=c(\data)$.
Moreover, by using the change of variables with $\bar Q_j=B_j\times (-\rho_j^q,0)$, $Y_j = |A_j| / |\bar Q_j|$, $A_{j} = \bar Q_j \cap \left\{ \bar u > k_j \right\}$ and $\bar u (\cdot, \bar t) = u(\cdot,t)$ for $\bar t = t/((4a_0)^{-1}\bomega_F^{2-q})$, we arrive at
$$
Y_{j+1}\le c2^{2qj}Y_{j}^{1+\frac{q}{n+q}},
$$
while $Y_0\le 2\nu_2$ holds by \eqref{669}. The proof of this case is completed owing to Lemma~\ref{lem:fgc}, provided that
\begin{align}\label{673}
    2\nu_2\le c^{-\frac{n+q}{q}}2^{-2q\frac{(n+q)^2}{q^2}}.
\end{align}

\end{proof}

Note that in the proof of the previous lemma, the scaling factor of the time interval $I_j$ is chosen with respect to the first conditions of \ref{5} while the Caccioppoli inequality is estimated by using the second condition of \ref{5}. Thus, in this case, we use the $q$-phase for both the scaling factor and the Caccioppoli inequality. On the other hand, the first condition in~\ref{6} is the $q$-phase criterion, while the second is of $p$-phase type. Since we have set $F^{q-2}=A$, it follows from \ref{6} that 
\begin{align*}
\begin{split}
    (A/2)\theta_{\rho/2}(\rho/2)^2
    &=(A/2)\left(\Bigl( \frac{\bomega}{\rho/2}\Bigr)^{p}+a_0\Bigl( \frac{\bomega}{\rho/2}\Bigr)^{q}\right)^{-1}\bomega^2\\
    &\ge (A/2)\left(2a_0\Bigl( \frac{\bomega}{\rho/2}\Bigr)^{q}\right)^{-1}\bomega^2\\
    &=(1/2)\left(2^{q+1}a_0\Bigl( \frac{\bomega_F}{\rho}\Bigr)^{q}\right)^{-1}\bomega_F^2\\
    &\ge(1/2)\left(2^{2p}K\Bigl( \frac{\bomega_F}{\rho}\Bigr)^{p}\right)^{-1}\bomega_F^2\\
    &\ge \Bigl( \frac{\bomega_F}{\rho/(8K)}\Bigr)^{-p}\bomega_F^2,
\end{split}    
\end{align*}
and therefore 
$$
    Q_{\rho/(8K),\bomega_F^{2-p}(\rho/(8K))^p}\subset Q_{\rho/2,(A/2)\theta_{\rho/2}(\rho/2)^2}.
$$
On the other hand, the above inclusion also holds for the case \ref{7} as we have
\begin{align*}
    \begin{split}
    (A/2)\theta_{\rho/2}(\rho/2)^2
    &=(A/2)\left(\Bigl( \frac{\bomega}{\rho/2}\Bigr)^{p}+a_0\Bigl( \frac{\bomega}{\rho/2}\Bigr)^{q}\right)^{-1}\bomega^2\\
    &\ge (A/2)\left(2^{q+1}K   \Bigl( \frac{\bomega}{\rho}\Bigr)^{p}     \right)^{-1}\bomega^2\\
    &\ge F^{p-2}\left(2^{3p}K   \Bigl( \frac{\bomega}{\rho}\Bigr)^{p}     \right)^{-1}\bomega^2\\
    &\ge \Bigl( \frac{\bomega_F}{\rho/(8K)}\Bigr)^{-p}\bomega_F^2.
    \end{split}
\end{align*}

\begin{lemma}\label{lem:degiorgi-second-alt2`}
Suppose either \textnormal{\ref{6}} or \textnormal{\ref{7}} holds. There exists $\nu_2(\data) \in (0,1)$ such that if 
\begin{align}\label{675}
    \bigl| \bigl\{ u \geq \muplus - \bomega_F \bigr\} \cap Q_{\rho/(8K),\bomega_F^{2-p}(\rho/(8K))^p}  \bigr|\leq \nu_2 \bigl| Q_{\rho/(8K), \bomega_F^{2-p}(\rho/(8K))^p} \bigr|
\end{align}
holds true, then
$$
u < \muplus-\frac{\bomega_F}{2} \quad\text{a.e. in}\quad \,  Q_{\rho/(16K), \bomega_F^{2-p}(\rho/(16K))^p}.
$$
\end{lemma}
\begin{proof}
For each $j\in\mathbb{N}$, we set
     $$
         k_j=\muplus-\frac{\bomega_F}{2}-\frac{\bomega_F}{2^{j+1}},
     $$
     and then we have
     $$
         (u-k_j)_+\le \bomega_F.
     $$
We denote
$$
    \rho_j=\frac{\rho}{16K}+\frac{\rho}{2^{j+4}K},\quad B_j=B_{\rho_j},\quad I_j=\bigl(-\bomega_F^{2-p}\rho_j^p,0\bigr),\quad Q_j=B_j\times I_j,
$$
when \eqref{675} can be written as
\begin{align}\label{375}
    \bigl| \bigl\{ u \geq \muplus - \bomega_F \bigr\} \cap Q_0 \bigr|\le \nu_2|Q_0|.
\end{align}
Let $0\le \varphi_j\le 1$ be a Lipschitz function such that $\varphi_j=1$ in $Q_{j+1}$ and $\varphi_j=0$ on   the parabolic boundary of $Q_j$   with
$$
    |\na\varphi_j|\le c\frac{2^{j}}{\rho},\quad |\pa_t\varphi_j|\le c \frac{2^{pj } \bomega_F^{p-2}}{\rho^p}.
$$
Arguing as in case~\ref{p} of the proof of Lemma~\ref{lem:degiorgi-first-alt`}, we get
\begin{align*}
    \begin{split}
    &\bomega_F^{2-p}\sup_{t\in I_0}\int_{B_0} (u-k_j)_+^p\varphi_j^q\, \d x + \iint_{Q_0} |\na ((u-k_j)_+\varphi_j^\frac{q}{p})|^p\, \d x \d t\\
    &\qquad\le c\frac{2^{qj} \bomega_F^p}{\rho^p}\iint_{Q_0}\chi_{\{Q_j: u>k_j\}}\, \d x \d t,
\end{split}
\end{align*}
where $c=c(\data)$.
As in the proof of Lemma~\ref{lem:degiorgi-first-alt`}, using the change of variables $\bar u (\cdot, \bar t) = u(\cdot,t)$ for $\bar t = t/\bomega_F^{2-p}$ and denoting $\bar Q_j=B_j\times (-\rho_j^p,0)$, $A_{j} = \bar Q_j \cap \left\{ \bar u > k_j \right\}$ and $Y_j=|A_j|/|\bar Q_j|$, we obtain from the embedding theorem that
$$
Y_{j+1}\le c2^{2qj}Y_{j}^{1+\frac{p}{n+p}}.
$$
Since $Y_0\le \nu_2$ holds by \eqref{375}, taking
$$
    \nu_2\le c^{-\frac{n+p}{p}}2^{-2q\frac{(n+p)^2}{p^2}}
$$
the proof is completed.
\end{proof}

\begin{lemma}\label{lem:A} 
    Suppose that~\eqref{eq:second-alt`} holds with $\nu_0(data) \in(0,1)$ determined in Lemma~\ref{lem:degiorgi-first-alt`}. There exists sufficiently large $A(\data)\ge4$, satisfying $4\le A^{\frac{1}{q-2}}\in\mathbb{N}$, such that if \textnormal{\ref{5}} holds, then \eqref{666} is satisfied and if either \textnormal{\ref{6}} or \textnormal{\ref{7}} holds, then \eqref{675} is satisfied.
\end{lemma}

\begin{proof}    
    \textit{Case}~\ref{5}.
    Let $D = D(data)$ be the constant from Lemma~\ref{lem:second-alt-expansion}, fix $E$ satisfying $\dd\le E\le F/2$ and consider $\bomega_E=\bomega/E$ and $k=\muplus-\bomega_E$.
    Also, we denote
    $$
        \tq=Q_{\rho, (A/2) \theta_\rho\rho^2},\quad Q=Q_{\rho/2, (A/2) \theta_{\rho/2}(\rho/2)^2}
    $$
     and let $0\le\varphi\le 1$ be a Lipschitz function such that $\varphi=1$ in $Q$ and $\varphi=0$ on   the parabolic boundary of $\tq$   with 
     \begin{align}\label{695}
         |\na\varphi|\le \frac{2}{\rho},\quad |\pa_t\varphi|\le \frac{c}{A\theta_\rho \rho^2},
     \end{align}
     where we used the fact that $\theta_{\rho/2}(\rho/2)^2\le \theta_\rho\rho^2\le 2^q\theta_{\rho/2}(\rho/2)^2$.
    Applying Lemma~\ref{lem:caccioppoli} and Lemma~\ref{phase_criterion}, we get
     \begin{align*}
         \begin{split}
             &\iint_{Q} a_0|\na (u-k)_+|^q\, \d x \d t  \\
			&\qquad\le c\iint_{\tq} (u-k)_+^p (|\na \varphi|^p+\rho^\alpha|\na\varphi|^q) +a_0(u-k)_+^q |\na \varphi|^q \, \d x \d t\\
			&\qquad\qquad+c\iint_{\tq} (u-k)_+^2 |\partial_t \varphi^q| \, \d x \d t
         \end{split}
     \end{align*}
    for some $c(\data)$. 
    Using \eqref{695}, $(u-k)_+\le \bomega_E$ and \ref{5}, we get
    \begin{align}\label{382}
        \iint_{Q} |\na (u-k)_+|^q\, \d x \d t \le c\Bigl(\frac{\bomega_E}{\rho}\Bigr)^q|Q|,
    \end{align}
    where we used \eqref{365} and $E\le F$ to estimate the last term of the Caccioppoli inequality. Moreover, applying H\"older's inequality, the above inequality becomes
    \begin{align}\label{683}
        \iint_{Q} |\na (u-k)_+|^p\, \d x \d t \le c\Bigl(\frac{\bomega_E}{\rho}\Bigr)^p|Q|=c\Bigl(\frac{\bomega}{E\rho}\Bigr)^p|Q|.
    \end{align}
    
    On the other hand, as we have set $\dd\le E$ it follows from Lemma~\ref{lem:second-alt-expansion} that
    \begin{align*}
        |B_{\rho}\cap \{u< \muplus-\tfrac{\bomega}{E}\}|\ge \frac{\nu_0^2}{4}|B_{\rho}| \quad\text{for a.e.}\quad t\in (-(A-1)\theta_\rho\rho^2,0).
    \end{align*}
    Thus, Lemma~\ref{lem:isoperimetric} with $k=\muplus-\tfrac{\bomega}{E}$ and $l=\muplus-\tfrac{\bomega}{2E}$, along with the above inequality, gives
    $$
        \frac{\bomega}{2E}|B_{\rho}\cap \{u>\muplus-\tfrac{\bomega}{F}\}|\le \frac{c\rho}{\nu_0^2}\int_{B_\rho\cap \bigl\{-\tfrac{\bomega}{E}<u-\muplus<-\tfrac{\bomega}{2E}\bigr\}}|\na u|\,\d x.
    $$
    Integrating over the time interval $(-(A/2)\theta_{\rho/2}(\rho/2)^2,0 )$ and applying Hölder's inequality, we get
    \begin{align*}
    \begin{split}
        &\frac{\bomega}{2E}|Q\cap \{u>\muplus-\tfrac{\bomega}{F}\}|\\
        &\qquad\le \frac{c\rho}{\nu_0^2} \left(\iint_{Q\cap \bigl\{\muplus-\tfrac{\bomega}{E}<u\bigr\}}|\na u|^p\,\d x\d t\right)^\frac{1}{p} |Q\cap \bigl\{-\tfrac{\bomega}{E}<u-\muplus<-\tfrac{\bomega}{2E}\bigr\}|^\frac{p-1}{p}.
    \end{split}
    \end{align*}
    We estimate the above integral using \eqref{683} to have
    $$
       |Q\cap \{u>\muplus-\tfrac{\bomega}{F}\}|^\frac{p}{p-1}\le c\nu_0^{-\frac{2p}{p-1}}|Q|^\frac{1}{p-1}|Q\cap \bigl\{-\tfrac{\bomega}{E}<u-\muplus<-\tfrac{\bomega}{2E}\bigr\}|.
    $$
    As the above inequality holds for $\dd\le  E\le F/2$, we sum over $\dd,2\dd,...,F/2$ to get
    $$
        (F/2-\dd)^\frac{p-1}{p}|Q\cap \{u>\muplus-\tfrac{\bomega}{F}\}|\le c\nu_0^{-2}|Q|.
    $$
    Therefore, \eqref{666} holds provided that
    $$
        \frac{c}{\nu_0^{2}(F/2-\dd)^\frac{p-1}{p}}\le \nu_2.
    $$
    Recalling \eqref{673}, the proof is completed in this case by taking $F$ large enough to satisfy the above inequality.

    \textit{Case}~\ref{6} or \ref{7}. 
    Let $E$ and $k=\muplus-\bomega_E$ be as in the previous case and set
    $$
        \tq=Q_{\rho/(4K),\bomega_F^{2-p}(\rho/(4K))^p},\quad Q=Q_{\rho/(8K),\bomega_F^{2-p}(\rho/(8K))^p}.
    $$
     Moreover, let $0\le\varphi\le 1$ be a Lipschitz function such that $\varphi=1$ in $Q$ and $\varphi=0$ on   the parabolic boundary of $\tq$   with 
     \begin{align}\label{681}
         |\na\varphi|\le \frac{2}{\rho},\quad |\pa_t\varphi|\le c\frac{\bomega_F^{p-2}}{\rho^p}.
     \end{align}
    In order to estimate the Caccioppoli inequality, we consider the following subcases:
    \begin{enumerate}[label=(\alph*),start=8]
        \item\label{8} $a_0\bigl(\tfrac{\bomega_E}{\rho}\bigr)^q\le K\bigl(\tfrac{\bomega_E}{\rho}\bigr)^p$ holds.
        \item\label{9} $a_0\bigl(\tfrac{\bomega_E}{\rho}\bigr)^q> K\bigl(\tfrac{\bomega_E}{\rho}\bigr)^p$ holds.
    \end{enumerate}
   Note that \ref{7} always implies \ref{8}. 
    
    \textit{Subcase}~\ref{8}.
    Applying Lemma~\ref{lem:caccioppoli}, we get
    \begin{align*}
         \begin{split}
             &\iint_{Q} |\na (u-k)_+|^p\, \d x \d t  \\
			&\qquad\le c\iint_{\tq} (u-k)_+^p (|\na \varphi|^p+\rho^\alpha|\na\varphi|^q) +a_0(u-k)_+^q |\na \varphi|^q \, \d x \d t\\
			&\qquad\qquad+c\iint_{\tq} (u-k)_+^2 |\partial_t \varphi^q| \, \d x \d t
         \end{split}
     \end{align*}
    for some $c =c(\data) > 0$. 
    Using \eqref{681}, $E\le F$, $(u-k)_+\le \tfrac{\bomega}{E}$ and \ref{8}, we get
    $$
        \iint_{Q} |\na (u-k)_+|^p\, \d x \d t \le c\Bigl(\frac{\bomega}{E\rho}\Bigr)^p|Q|,
    $$
   which is the same as \eqref{683}. Thus, the conclusion in this case follows as in the case~\ref{5}.

    \textit{Subcase}~\ref{9}.
    Applying Lemma~\ref{lem:caccioppoli} and Lemma~\ref{phase_criterion}, we get
       \begin{align*}
         \begin{split}
             &\iint_{Q} a_0|\na (u-k)_+|^q\, \d x \d t  \\
			&\qquad\le c\iint_{\tq} (u-k)_+^p (|\na \varphi|^p+\rho^\alpha|\na\varphi|^q) +a_0(u-k)_+^q |\na \varphi|^q \, \d x \d t\\
			&\qquad\qquad+c\iint_{\tq} (u-k)_+^2 |\partial_t \varphi^q| \, \d x \d t
         \end{split}
     \end{align*}
    for some $c =c(\data) > 0$. As $(u-k)_+\le \tfrac{\bomega}{E}$ holds, it follows from \eqref{681} and \ref{9} that
    $$
             \iint_{Q} a_0|\na (u-k)_+|^q\, \d x \d t \le ca_0\Bigl(\frac{\bomega}{E\rho}\Bigr)^q|Q|,
    $$
   where to estimate $|\pa_t\varphi^q|$, we used the fact that by $E\le F$ and \ref{9} there holds
    $$
        \Bigl(\frac{\bomega}{F\rho}\Bigr)^{p-2}\le \Bigl(\frac{\bomega}{E\rho}\Bigr)^{p-2}<a_0\Bigl(\frac{\bomega}{E\rho}\Bigr)^{q-2}.
    $$
    Therefore, we obtained \eqref{382} and the rest of the proof is the same.

    As we have covered all cases, the proof is completed.
\end{proof}

Finally, we conclude this subsection by the reduction of oscillation when the second alternative~\eqref{eq:second-alt`} holds.

\begin{corollary}\label{cor:alt-second-conclusion`}
        Suppose that~\eqref{eq:second-alt`} holds with $\nu_0 = \nu_0(data) \in(0,1)$ determined in Lemma~\ref{lem:degiorgi-first-alt`} and let $A = A (data) \ge4$ be as in Lemma~\ref{lem:A}. Then there holds
    $$
        \osc_{Q_{\la\rho,\theta_{\la\rho}(\la\rho)^2}}u\le \left(1-\frac{1}{2F}\right)\bomega,
    $$
    where $\la=\tfrac{1}{16K}$.
\end{corollary}

\begin{proof}
    Since $A\ge4$ holds, for the case \ref{5} it follows from Lemmas~\ref{lem:A} and~\ref{lem:degiorgi-second-alt1`} that
    $$
        \osc_{Q_{\la\rho,\theta_{\la\rho} (\la\rho)^2}}u <\left(1-\frac{1}{2F}\right)\bomega.
    $$
    On the other hand, we apply Lemmas~\ref{lem:A} and~\ref{lem:degiorgi-second-alt2`} to deal with cases \ref{6} and \ref{7}. Since $A^\frac{p-2}{q-2}=F^{p-2}\ge1$ and
    $$
        \bomega_F^{2-p}(\la\rho)^p=F^{p-2}\bomega^{2-p}(\la\rho)^p\ge \theta_{\la\rho}(\la\rho)^2,
    $$
    we have the conclusion.
\end{proof}

\subsection{Recursive argument}

In the following corollary we have combined the conclusions in Corollary~\ref{cor:alt-first-conclusion`} and Corollary~\ref{cor:alt-second-conclusion`}. Note that the estimate below holds also when \eqref{eq:omega-alternative`} is false.
\begin{corollary}\label{reduction of osc}
    Let $A(\data)\ge4$ be as in Lemma~\ref{lem:A}. If \eqref{62} holds, then there exists $\varepsilon(\data)\in[\tfrac{1}{2},1)$ satisfying
    $$
        \osc_{Q_{\la\rho,\theta_{\la\rho}(\la\rho)^2}}u\le \varepsilon\bomega,
    $$
    where $\la=\tfrac{1}{16K}$ with $K=\max \{1, 4[a]_\alpha\|u\|_\infty^{q-p}\}$.
\end{corollary}

\begin{corollary}\label{recursive}
        Let $A(\data)\ge4$ be as in Lemma~\ref{lem:A} and $0<\rho<\min\{1,2A^{-\frac{1}{p-2}}\|u\|_\infty\}$. Then there exist $c(\data)$ and $\beta(\data)\in(0,1)$ such that for $r\in(0,\rho)$
    $$
            \osc_{Q_{r,\theta_{(2\|u\|_\infty,r)} r^2}}u\le c\|u\|_\infty\Bigl(\frac{r}{\rho}\Bigr)^\beta,
    $$
    where 
    $$
    \theta_{(\kappa, r)} =\left(\Bigl( \frac{\ka}{r} \Bigr)^p+a_0\Bigl( \frac{\ka}{r} \Bigr)^q\right)^{-1}\Bigl( \frac{\ka}{r} \Bigr)^2.
     $$
\end{corollary}

\begin{proof}
We set $\bomega_0 =2\|u\|_{\infty}$ and $\theta_r=\theta_{(\bomega_0,r)}$.
By using induction, we claim that for every $j\in\mathbb{N}$ there holds
        \begin{align}\label{6103}
           A^\frac{1}{p-2}\rho_j<\bomega_j  \quad \text{and}\quad  \osc_{Q_{j}}u\le  \bomega_j,
        \end{align}
        where $\bomega_j = \eps^j \bomega_0$ with $\varepsilon\in[\tfrac{1}{2},1)$ defined in Corollary~\ref{reduction of osc} and
    $$
      \rho_j = \delta^j \rho,\quad   \delta=(16KA^\frac{1}{p})^{-1}\varepsilon^{1-\frac{2}{p}},\quad Q_{j}=Q_{\rho_j,\theta_{(\bomega_j,\rho_j)}\rho_j^2}.
    $$

    We define $\muplus_0 = \sup_{Q_{\rho,\rho^2}} u$ and $\muminus_0 = \inf_{Q_{\rho,\rho^2}} u $. Clearly assumptions in the beginning of Section~\ref{sec:reduction-osc} are satisfied and~\eqref{6103} holds in case $j=0$. For $j \in \mathbb{N}_{\geq 1}$, we define 
$$
\muplus_j = \sup_{Q_{\rho,A \theta_{(\bomega_j,\rho_j)}\rho_j^2}} u \quad \text{and}\quad \muminus_j = \inf_{Q_{\rho,A \theta_{(\bomega_j,\rho_j)}\rho_j^2}} u. 
$$
Then suppose that~\eqref{6103} holds for some $j \in \mathbb{N}$. Now
$$
       \rho_{j+1} = \delta \rho_j < \delta A^{-\frac{1}{p-2}} \bomega_j \le  \frac{2}{16KA^\frac{1}{p}} A^{-\frac{1}{p-2}} \varepsilon \bomega_j \le \frac{\bomega_{j+1}}{A^\frac{1}{p-2}},
$$
    since $\eps \geq \tfrac{1}{2}$. Thus,~\eqref{6103}$_1$ holds for index $j+1$ and by Corollary~\ref{reduction of osc} we obtain 
$$
 \osc_{Q_{\la\rho_j,\theta_{(\bomega_j,\la\rho_j)}(\la\rho_j)^2}}u\le \varepsilon \bomega_j = \bomega_{j+1},
$$
where $\la=\tfrac{1}{16K}$. Observe that we have  $\theta_{(\kappa,\la r)}(\la r)^2\ge A\theta_{(\varepsilon\kappa,\delta r)}(\delta r)^2$, since $A\varepsilon^2\ge1$ implies
     \begin{align*}
        \begin{split}
            &\left(\biggl( \frac{(A\varepsilon^2)^{-\frac{1}{p}}\varepsilon \kappa}{(A\varepsilon^2)^{-\frac{1}{p}}\varepsilon\la r} \biggr)^p+a_0\biggl( \frac{(A\varepsilon^2)^{-\frac{1}{p}}\varepsilon\kappa}{(A\varepsilon^2)^{-\frac{1}{p}}\varepsilon\la r} \biggr)^q\right)^{-1}\kappa^2\\
            &\qquad\ge A\left(\biggl( \frac{\varepsilon \kappa}{(A\varepsilon^2)^{-\frac{1}{p}}\varepsilon\la r} \biggr)^p+a_0\biggl( \frac{\varepsilon \kappa}{(A\varepsilon^2)^{-\frac{1}{p}}\varepsilon\la r} \biggr)^q\right)^{-1}(\ep\kappa)^2.
        \end{split}
    \end{align*}
Thus, $Q_{\rho_{j+1},A \theta_{(\bomega_{j+1}, \rho_{j+1})}\rho_{j+1}^2} \subset Q_{\la\rho_j,\theta_{(\bomega_j,\la\rho_j)}(\la\rho_j)^2}$ and 
$$
 \osc_{Q_{\rho_{j+1}, A \theta_{(\bomega_{j+1}, \rho_{j+1})}\rho_{j+1}^2}}u\le \bomega_{j+1},
$$
which implies that~\eqref{6103} holds for all $j \in \mathbb{N}$ by induction. 

Observe that $\theta_{(\varepsilon\kappa,r)}r^2>\theta_{(\kappa,r)}r^2$ holds since we have 
  \begin{align*}
  \begin{split}
      &\left(\Bigl( \frac{\varepsilon\ka}{r} \Bigr)^p+a_0\Bigl( \frac{\varepsilon\ka}{r} \Bigr)^q\right)^{-1}(\varepsilon\ka)^2\\
      &\qquad=\left(\frac{1}{\varepsilon^2}\Bigl( \frac{\varepsilon\ka}{r} \Bigr)^p+a_0\frac{1}{\varepsilon^2}\Bigl( \frac{\varepsilon\ka}{r} \Bigr)^q\right)^{-1}\ka^2\\
      &\qquad\ge \left(\Bigl( \frac{\ka}{r} \Bigr)^p+a_0\Bigl( \frac{\ka}{r} \Bigr)^q\right)^{-1}\ka^2.
  \end{split}
  \end{align*}
  Therefore, by \eqref{6103} we have
  $$
    \osc_{\tq_j }u\le \bomega_j,
  $$
  where
  $$
     \tq_j= Q_{\delta^j\rho, \theta_{(\bomega_0,\delta^j\rho)} (\delta^j\rho)^2}=Q_{\delta^j\rho,\theta_{\delta^j\rho} (\delta^j\rho)^2}.
  $$
  By setting $\beta=\tfrac{\log \varepsilon}{\log \delta}\in(0,1)$, there exists $c(\data)$ such that
   $$
            \osc_{Q_{r,\theta_{(\bomega_0,r)}r^2}}u\le c\bomega_0\left(\frac{r}{\rho}\right)^\beta.
    $$
    This completes the proof.
\end{proof}

Finally, we are ready to prove our main theorem.
\begin{proof}[Proof of Theorem \ref{thm}]
We replace the pointwise scaling factor with a uniform scaling factor. For any $r\in(0,2\|u\|_\infty)$, there holds
$$
    \left(\frac{2\|u\|_{\infty}}{r}\right)^{p}+a(x,t)\left(\frac{2\|u\|_{\infty}}{r}\right)^{q}<(1+\|a\|_{\infty})\left(\frac{2\|u\|_{\infty}}{r}\right)^{q}.
$$
Therefore, by applying Corollary~\ref{recursive} we get
$$
            \osc_{Q_{r,\Theta r^q}(x_0,t_0)}u\le c\|u\|_\infty\Bigl(\frac{r}{\rho}\Bigr)^\beta
 $$
    for $0<r<\rho<\min\{1,2 A^{-\frac{1}{p-2}}\|u\|_\infty\}$ with $Q_{\rho,\Theta\rho^q}(x_0,t_0)\subset \Om_T$ where $\Theta=(1+\|a\|_\infty)^{-1}(2\|u\|_{\infty})^{2-q}$. Moreover, as $\dist_q(\cdot,\cdot)$ is comparable with this metric by replacing $\|u\|_\infty$ with $2\|u\|_\infty$, the proof is completed.

\end{proof}

\section*{Acknowledgement}
 K. Moring has been supported by the Magnus Ehrnrooth Foundation. L. Särkiö has been supported by a doctoral training grant from Vilho, Yrj\"o and Kalle V\"ais\"al\"a Foundation.

\end{document}